\documentclass{article}
\usepackage{graphicx} 

\usepackage{latexsym,amsmath,amsfonts,amssymb,amstext,theorem,euscript,array,enumerate,mathrsfs}
\usepackage{color}
\usepackage{comment}
 
\newtheorem{Theorem}{Theorem}[section]

\newtheorem{Proposition}{Proposition}[section]
\newtheorem{Assumption}{Assumption}[section]
\newtheorem{Lemma}{Lemma}[section]

\newtheorem{Remark}{Remark}[section]
\newtheorem{Example}{Example}[section]
\numberwithin{equation}{section}

\def\esssup_#1{\underset{#1}{\mathrm{ess\,sup\, }}}
\def\essinf_#1{\underset{#1}{\mathrm{ess\,inf\, }}}

\def\qed{{\hfill\hbox{\enspace${ \square}$}} \smallskip}
\def\sqr#1#2{{\vcenter{\vbox{\hrule height .#2pt \hbox{\vrule
 width .#2pt height#1pt \kern#1pt \vrule
width .#2pt} \hrule height .#2pt}}}}
\def\square{\mathchoice\sqr54\sqr54\sqr{4.1}3\sqr{3.5}3}

\def\ds{\begin{displaystyle}}
\def\eds{\end{displaystyle}}

\def\<{\langle }
\def\>{\rangle }

\def \R{\mathbb{R}}

\def \E{\mathbb{E}}
\def \F{\mathbb{F}}

\def \P{\mathbb{P}}

\def \H{\mathbb{H}}

\def\cala{{\cal A}}
\def\calb{{\cal B}}

\def\calf{{\cal F}}

\def\calh{{\cal H}}

\def\calp{{\cal P}}

\def\call{{\cal L}}
\def\cals{{\cal S}}

\def\beqs{\begin{eqnarray*}}
\def\enqs{\end{eqnarray*}}
\def\beq{\begin{eqnarray}}
\def\enq{\end{eqnarray}}

\allowdisplaybreaks

\title{Partially observed controlled Markov chains \\and optimal control of the Wonham filter}

\author{Fulvia Confortola\footnote{Dipartimento di Matematica, Politecnico di Milano, {\sf fulvia.confortola at polimi.it} This author is a member  of INdAM-GNAMPA.}
\and Marco Fuhrman \footnote{Dipartimento di Matematica, Università degli Studi di Milano, {\sf marco.fuhrman at unimi.it} This author is a member  of INdAM-GNAMPA.} 
}

\date{}

\begin{document}

\maketitle

\begin{abstract}
We consider a class of optimal control problems, with finite or infinite horizon, for a continuous-time Markov chain with finite state space. In this case, the control process affects the transition rates. We suppose that the controlled process can not be observed, and at any time the  control actions are chosen based on the observation of a related stochastic process perturbed by an exogenous Brownian motion. We describe a construction of the controlled Markov chain, having stochastic transition rates adapted to the observation filtration.  By a change of probability measure of Girsanov type,  we introduce the so-called separated optimal control problem, where the  state is the conditional (unnormalized) distribution of the controlled Markov chain and the observation process becomes a driving Brownian motion, and we prove the equivalence with the original control problem. The controlled equations for the  separated problem are an instance of the Wonham filtering equations.   Next we present an analysis of the separated problem: we characterize the value function as the unique viscosity solution to the dynamic programming equations  (both in the parabolic and the elliptic case) we prove verifications theorems and a version of the stochastic maximum principle in the form of a necessary conditions for optimality.
\end{abstract}


\vspace{5mm}


\vspace{5mm}


\section{Introduction}

This paper is devoted to the study of optimal control problems for controlled Markov chains  with partial observation. Except for some initial general constructions, we will consider  controlled Markov processes $(X^\alpha_t)_{t\ge0}$ which are time-continuous  and with values in a finite state space $S$.
The controlled process depends on a control process $(\alpha_t)$, with values in a general action space $A$, which is chosen in order to maximize a reward functional of the form
$$
J(\alpha)=\bar\E\left[ \int_0^T 
f(X^\alpha_t,\alpha_t)\,dt   + 
g(X^\alpha_T)
\right],\qquad
\text{\rm or }
\qquad
J(\alpha)=\bar\E\left[ \int_0^\infty
e^{-\beta t}f(X^\alpha_t,\alpha_t)\,dt\right],
$$
for the finite and
infinite horizon cases, 
where $f$, $g$ are given real functions and
 $\beta >0$ is a discount factor (below we also consider  some slightly 
 more general reward functionals).
Here $\bar \E$ denotes the expectation with respect to some probability $\bar\P$, called the ``physical'' probability to distinguish it from the reference probability $\P$ introduced below.
 
We consider the case of partial observation, namely when the state is not directly observable and the choice of the control $\alpha_t$ at any time $t$ is  based on the observation of the past values of another related process, denoted $(W_t)_{t\ge0}$. In the literature the related terminology Hidden (or Latent) Markov Model  is also used.  Thus, the control process $(\alpha_t)$ will be required to be $(\calf^W_t)$-predictable, where $(\calf^W_t)$ is the $\sigma$-algebra  generated by $(W_t)$.
In our model we assume that the observation process $W$ takes values in $\R^d$ and has the form
\begin{equation}
    \label{obsintro}
W_t=\int_0^th(X_s^\alpha,\alpha_s)\,ds+B_t
\end{equation}
where $h:S\times A\to \R^d$ is a given function and $(B_t)_{t\ge0}$ is a Brownian motion in $\R^d$. Among many possible variations, this model - controlled Markov chain with observation corrupted by Brownian noise - is often deemed to be of basic importance. 

The main route to the solution of the optimal control problem - that we also adopt in this paper - consists in reducing it to a different problem with complete observation (sometimes called the separated problem) where the controlled state process is given by the so-called filter process, whose values at time $t$ are conditional distributions of the unobserved process $X_t^\alpha$ given $\calf^W_t$. For our model, in the uncontrolled case, explicit recursive equations for the filter were obtained in \cite{Wo1965} and their solutions are called the Wonham filter.

There is a huge literature on partially observed control problems and 
we  refer the reader to the monographs \cite{Bebook}, \cite{nisio2015} and \cite{ElAgMobook} which include expositions of the required technical prerequisites and  contain 
extensive references. 
The books \cite{Bebook} and \cite{nisio2015} mainly consider the case when the controlled process is defined as the solution to a controlled stochastic differential equation in Euclidean space driven by a Brownian motion. The treatise \cite{ElAgMobook} presents a large number of hidden Markov models with many variations with respect to our case, for instance discrete-time problems, continuous state spaces, different observation models and so on. In the sequel we will also refer to \cite{BaCribook}
and \cite{Brbook}, dealing with technical aspects on stochastic filtering theory and optimal control of marked point processes.

The analysis of our model is of course made easier by the assumption that the state space $S$ is finite, but it turns out that  a direct application of  general existing theories does not yield  satisfactory results, as it requires unnecessary assumptions or it does not give sharp conclusions. It is the purpose of this paper  to present a rather complete analysis of the model sketched above, with various methodologies (stochastic maximum principle and dynamic programming, including analysis of the Hamilton-Jacobi-Bellman equation), encompassing the finite and infinite horizon case and with a careful formulation of the optimization problem.
Except for some natural boundedness or continuity assumptions on the coefficients (the functions $f$, $g$, $h$ introduced above, as well as the controlled transition rates presented below) we try to be as general as possible.

In order to explain more carefully our contributions we have to enter some technical details while we describe the plan of the paper at the same time. The first issue concerns the construction of a controlled Markov chain. In this case the transition rate from state $i\in S$ to state $j\neq i$, denoted $q(a,i,j)$, depends on the choice of the control parameter $a\in A$. Given the functions 
$q(a,i,j)$ and
an $\F^W$-predictable control process $(\alpha_t)$ the aim is to construct a process $(X^\alpha_t)$ admitting stochastic transition rates $q(\alpha_t,i,j)$. The precise meaning of this, according to most of the literature,  is that 
 the random measure $q(\alpha_t,X^\alpha_{t-},j)dt$ is the compensator of the process 
$N_t(j)$ which counts the number of jumps of $(X^\alpha_t)$ to the state $j$ in the time interval $[0,t]$, namely
$$
N_t(j)-\int_0^t q(\alpha_s,X^\alpha_{s-},j)ds
\quad \text{is a martingale}
$$
 with respect to the filtration generated by $(W_t)$ and by the controlled process itself. 
When there is no observation process and the only filtration is the natural one, the existence of the controlled process may be deduced from a general result on a martingale problem for marked point process: see \cite{Ja1974}. In this case the controlled process is defined in a weak sense, as a law on a canonical space. In the general case with observation, when the state $S$ is finite, one can write down stochastic differential equations for a pure jump process identifying $S$ with a finite subset of $\R^N$: see \cite{ElAgMobook} chapter 12. In the present paper we revert to a different construction which is inspired by the Grigelionis theorem (see e.g. \cite{Brbook} section 5.7). It admits several variants: see for instance Section 3 of \cite{BrMa96} for related results and references.
We construct the controlled process in strong formulation, starting from an auxiliary Poisson process on an extended space and then taking an appropriate projection (depending on the control process) on $(0,\infty)\times S$ of the corresponding random measure. This direct construction for a controlled pure jump process  has the advantage that it can be extended to general state space $S$.  Section 
 \ref{seccontruction} is devoted to the exposition of this result in its general form.

In the following sections we apply the previous construction and we formulate the optimal control problem. In order to introduce the separated control problem for the Wonham filter one needs to perform a change of probability of Girsanov type: given the martingale
$$
(Z^\alpha_t)^{-1}=\exp\left(-\int_0^t h(X^\alpha_s,\alpha_s)\,dB_s-\frac{1}{2} \int_0^t |h(X^\alpha_s,\alpha_s)|^2\,ds\right),
\qquad t\ge 0,
$$
(this involved notation is consistent with the following sections) 
one defines the so-called reference probability $\ \P$ setting $d\P(d\omega)= Z^\alpha_T(\omega)d\bar\P(d\omega)$
and the filter process of
 the unnormalized conditional laws
$$
\rho_t^i =\E[1_{X_t^\alpha=i}\,Z^\alpha_t\,|\, \calf_t^W], \qquad t\ge 0,i\in S.
$$
By the Girsanov theorem the observation $W$ is a Brownian motion on $[0,T]$ under $ \P$ and
it is well known (see e.g. \cite{BaCribook}) that the processes $(\rho_t^i)$ solve the 
Zakai filtering equations, that are called the Wonham filtering equations in this particular situation:
\begin{equation}
    \label{sdewonham}
d\rho_t^i= \sum_{j\in S} \rho_t^j q(\alpha_t,j,i)\,dt
+ \rho_t^i\,h(i,\alpha_t)\,dW_t,
\qquad i\in S.
\end{equation}
The reward functional   takes the form
\begin{equation}
    \label{rewardwonham}
J(\alpha)=\E\left[ \int_0^T  \sum_{i\in S}
\rho_t^i\,f(i,\alpha_t)\,dt  + \sum_{i\in S}\rho_T^i\,g(i)
\right],\qquad
\text{\rm or }
\qquad
J(\alpha)=\E\left[ \int_0^\infty e^{-\beta t}
\sum_{i\in S}
\rho_t^i\,f(i,\alpha_t)\,dt\right].
\end{equation}
This way we obtain 
the separated control problem, where the new state equation is now \eqref{sdewonham},  driven by the observation Brownian motion $(W_t)$, so that the new control problem is fully observed. As it is customary (see e.g. \cite{Bebook}) it is more convenient to formulate the entire setting under the reference probability $\P$ from the beginning and to perform the inverse Girsanov transformation to construct the physical probability $\bar\P$: this way one obtains a weak formulation of the original control problem under $\bar\P$.
Section \ref{secformulazionepb} is devoted to the presentation of this standard material, and it also contains some preliminary properties of the corresponding value function.

In the following sections we address the optimal control problem for the state equation 
\eqref{sdewonham} and the reward
\eqref{rewardwonham}. The controlled state $\rho_t=(\rho_t^i)_{i\in S}$ evolves in the state space 
$$
 D=\{x=(x_1,\ldots,x_N)\in \R^N\,:\, x_i\ge0,\, i=1,\ldots,N\}.
 $$
We first consider the dynamic programming approach. We introduce the value functions $v(t,x)$ or $v(x)$ for the finite or infinite horizon cases, where $x\in D$ denotes the starting state. The value functions are related  to the Hamilton-Jacobi-Bellman (HJB) equations, which are, respectively, parabolic and elliptic equations on $D$.
For instance, in  the elliptic case 
 for a function
 $v(x)=v(x_1,\ldots,x_N)$ this is: 
\begin{equation}
    \label{hjbintro} 
\beta v(x)   -\sup_{a\in A}\bigg[
\frac{1}{2}\sum_{ij}\partial_{ij}^2v(x )\,
x_ix_j\sum_{k=1}^dh_k(i,a)h_k(j,a)
+\sum_{ij} \partial_iv(x)x_jq(a,j,i)
+ 
\sum_i
x_i\,f(i,a)  
\bigg]     =0.
\end{equation}
In the general case this equation is fully nonlinear and it is not uniformly elliptic, so the convenient notion of solution is the concept of viscosity solution, see e.g. \cite{CrIsLi1992}. While proving that the value function is a viscosity solution follows from standard results,  uniqueness of solutions is more delicate and is usually proved via comparison results between sub- and super-solutions to the equation. While there exists very sophisticated version of this kind of results for more general cases, for instance even when $D$ is replaced by a Hilbert space (see \cite{li1989}, \cite{GoSw2000} or \cite{FaGoSwbook}) we are not aware of any result which can be applied to \eqref{hjbintro}
or to its parabolic version, under our assumptions. Therefore we present two comparison theorems in Sections
\ref{sechjbellitt} and
\ref{sechjbparab},
thus establishing uniqueness of viscosity solution and concluding that the value functions are completely characterized analytically as solutions to suitable PDEs. In Section \ref{sec:verif} we prove two verification theorems, for the finite and infinite time horizon, showing that if a classical solution to the dynamic programming equation exists then, under some additional conditions, it coincides with the value function and it is possible to construct an optimal control in feedback form for the separated problem. In our context of a controlled finite Markov chain it may happen that the HJB equation is uniformly elliptic and analytical results on the existence of smooth solutions apply: see Remark \ref{equnifellitt}.

Section 
\ref{secSMP} is devoted to the approach to the control problem
\eqref{sdewonham}-\eqref{rewardwonham}
by means of the stochastic maximum principle.
This is a basic tool in stochastic optimization and as such it has been applied to partially observed optimal control problems as well. The reader may find an exposition and further references in \cite{Bebook} or \cite{ElAgMobook}. We formulate a stochastic maximum principle for the separated problem as a necessary condition for optimality related to our optimization problem. Although the proof relies on classical arguments, the final statement  improves existing results in the literature. 
Indeed,   the maximum principle for the controlled Zakai equation is usually formulated under the assumptions that the set of control a actions $A$ is convex and the coefficients are differentiable with respect to $a\in A$. These restrictions, for completely observable control problems, have been removed by Peng \cite{Pe1990} and we follow the same approach here. In spite of the greater generality, the final formulation does not require the second adjoint equation introduced in \cite{Pe1990}, since simplifications occur due to the linearity of the separated control problem with respect to the state variable. In any case, the restrictions mentioned above can be avoided and in particular the coefficients are only assumed to be continuous with respect to $a\in A$: see Theorem \ref{thSMP} below.

In conclusion, our contribution consists in the construction of a controlled Markov chain with stochastic transition rates adapted to a general given filtration (in particular, a Brownian filtration); the formulation of a separated optimal control problem for the Wonham filter and the proof of its equivalence with the original one; a complete and largely self-contained analysis of the separated problem, both for the finite and infinite horizon case, including a characterization of the value function as the unique viscosity solution to the dynamic programming equations, a verification theorem, an instance of the stochastic maximum principle in the form of a necessary condition for optimality.

\section{A construction of a point process with random compensator}
\label{seccontruction}

In this section we suppose that $S$ is a Polish space with a Borel probability measure $\mu$. We assume we are given a nonnegative function $q(\omega,t,x,y)$ with suitable properties (in particular, bounded) and we show how to construct an $S$-valued pure jump process $(X^q_t)$ such that the corresponding random measure admits compensator $q(t,X^q_{t-},y)\,\mu(dy)\,dt$. 
We refer e.g. to \cite{Brbook} for prerequisites on random measures and point processes.
In the following sections 
this construction will be applied to define a controlled Markov chain in $S$.
Our setting is summarized in the following hypotheses. 
\begin{Assumption}\label{settingprobabgen}
Assume that 
on a probability space $(\Omega,\calf,\P)$ the following independent random elements are defined:
\begin{enumerate}
    
    \item a Poisson process $(T_n)_{n\ge 1}$ on $(0,\infty)$ with intensity $K>0$; we set $T_0=0$;
    \item an independent sequence $(X_n)_{n\ge 1}$ of random variables, taking values in a  Polish space $S$, each with the same law $\mu$;
    \item  an   $S$-valued  random variable $X_0$;
    \item an independent sequence $(U_n)_{n\ge 1}$ of random variables, each uniformly distributed on $(0,1)$.
\end{enumerate}
\end{Assumption}

We define a random measure $\bar N(dt,dy,du)$ on $(0,\infty)\times S\times (0,1)$ by the formula
$$
\bar N(dt,dy,du)=\sum_{n\ge 1} \delta_{(T_n,X_n,U_n)}(dt,dy,du)
$$
and we denote $\F^{\bar N}=(\calf^{\bar N}_t)_{t\ge0}$ the filtration generated by $\bar N$ and $X_0$.
We note that $\bar N$ is a marked Poisson processes, with independent marks $(X_n,U_n)$ taking values in $S\times (0,1)$. Therefore the  $\F^{\bar N}$-compensator of $\bar N$ is
$$
\bar \nu(dt,dy,du)=K\, dt\, \mu(dy)\,du.
$$

Now  suppose that we are given a filtration $\F^1=(\calf^1_t)_{t\ge0}$   in $(\Omega,\calf)$, with $\calf^1_\infty$ independent from the above random processes and variables. Denote $\F^{\bar N,1}=(\calf^{\bar N,1}_t)_{t\ge0}$ the filtration defined by $ \calf^{\bar N,1}_t= \calf^{\bar N}_t\vee \calf^1_t$.
Since $\calf^1_\infty$ is independent of $\bar N$, it is easily verified that $\bar\nu$ is also the compensator of $\bar N$ with respect to  $\F^{\bar N,1}$.

Also  suppose that   we are given a function $q:\Omega\times[0,\infty)\times S\times S\to \R$
 satisfying $\P$-a.s.
\begin{equation}
    \label{qbdd}
0\le q(\omega,t,x,y)\le C_q,
\qquad t\ge 0, \;x,y\in S
\end{equation}
for some constant $C_q>0$.
We assume that $q$ is $\calp(\F^{1})\otimes \calb(S)\otimes \calb(S)$-measurable, where $\calp(\F^{1})$ denotes the predictable $\sigma$-algebra in $\Omega\times [0,\infty)$ for the filtration $\F^{1}$ and $\calb(S)$ the Borel $\sigma$-algebra in $S$.
Finally we
assume that the constant in Assumption \ref{settingprobabgen}-${\it 1}$ satisfies $K\ge C_q$.

Define inductively $\nu_0=0$ and, for $k\ge 0$,
$$
\nu_{k+1}=\inf\{ n>\nu_k \,:\,
U_n< q(T_n,X_{\nu_k},X_n)/K\},
$$
with the convention $\inf \emptyset =\infty$.
We take an element $\delta\notin S$ and we add it to $S$ as an isolated point. We set 
$T_{\nu_n}=\infty$ and $X_{\nu_n}=\delta$ if ${\nu_n}=\infty$ and
we 
consider the marked point process  $(T_{\nu_n},X_{\nu_n})_{n\ge 1}$. We also introduce  the corresponding
$S\cup\{\delta\}$-valued piecewise-constant process $(X^q_t)_{t\ge 0}$ (starting from $X_0$ at time $0
$) and the associated random measure $ N(dt,dy)$ on $(0,\infty)\times S$: 
$$\hbox{for } n\ge0,\quad
X^q_t= X_{\nu_n},\;
T_{\nu_n}\le t< T_{\nu_{n+1}};\qquad
  N(dt,dy )=\sum_{n\ge 1} \delta_{(T_{\nu_n},X_{\nu_n})}(dt,dy)\, 1_{T_{\nu_n}<\infty}.
$$
We denote $\F^{ N}=(\calf^{ N}_t)_{t\ge0}$ the filtration generated by $ N$ and $X_0$,    and by $\F^{ N,1}=(\calf^{ N,1}_t)_{t\ge0}$ the filtration defined by $ \calf^{ N,1}_t= \calf^{ N}_t\vee \calf^1_t$.
We note that in fact also $\nu_k$, $N(dt,dy)$ and $\F^{ N,1}$ depend on $q$, but we omit indicating this dependence.

\begin{Lemma}\label{Xqpredictablegen}
    The process $X^q$ is càdlàg and $\F^{\bar N,1}$-adapted.
\end{Lemma}

\noindent
{\bf Proof.}
Since
\begin{equation}
    \label{rapprXqgen}
X^q_t= \sum_{n\ge 0}X_{\nu_n}\,1_{[T_{\nu_n},T_{\nu_{n+1}})}(t) 1_{{\nu_n}<\infty}    
\end{equation}
we see that $X^q$ is clearly càdlàg. Adaptedness
  is intuitive, since at any time $t$ all its present and past values
and jump times  can be recovered observing $T_n$ and $q(T_n,i,j)$ up to time $t$ as well as the corresponding  $X_n,U_n$. Now we proceed to a formal proof.

\medskip 

{\it Step I: for each $k,n\ge 0$ 
$$\{\nu_k=n\}\in \calf^{\bar N,1}_{T_n},
$$ 
i.e., $\nu_k$ is a stopping time for the filtration 
$(\calf^{\bar N,1}_{T_n})_{n\ge0}$
}.

Since $T_n$ is a stopping time for $\F^{\bar N}$, it is also a stopping time for $\F^{\bar N,1}$. Since $(q(t,i.j))_t$
 is predictable for
 $\F^{1}$, it is also  predictable - and hence progressively measurable - for $\F^{\bar N,1}$. It follows that  $q(T_n,i.j)$ is $\F^{\bar N,1}_{T_n}$-measurable. The same holds for $(U_n,X_n)$ and hence for $q(T_n,i.X_n)$, by composition.  We define a discrete time filtration $\H= (\calh_n)_{n\ge 0}$ and, 
 for every $i\in S$, a time discrete process $(Y_n(i))_{n\ge0}$ by the formulae
 $$
Y_n(i)= (U_n,q(T_n,i.X_n)),\qquad \calh_n = 
 \F^{\bar N,1}_{T_n}
 $$
(here we set $U_0=0$). Then we have seen  that $(Y_n(i))_n$ is $\H$-adapted. It takes values in the set $\{(q,u)\,:\, 0<u<1, 0\le q <\infty\}$. 
Define  $D=\{(u,q)\,:\, u<q/K\}$. 
 
 We can express $\nu_1$ as the first hitting time of $D$ by the process $(Y_n(X_0))_n$: 
$$
\nu_{1}=\inf\{ n>0 \,:\,
U_n< q(T_n,X_{0},X_n)/K\}=
\inf\{ n>0 \,:\,
Y_n(X_0) \in D\}.
$$
Since $(Y_n(X_0))_n$
  is $\H$-adapted we conclude that $\nu_1$ is a stopping time for  $\H$.
Since $(X_n)_n$ is $\H$-adapted, the process $(X_{n\wedge \nu_1})_n$ is also $\H$-adapted.

Similarly, we can express $\nu_1$ and $\nu_2$ as the first and second hitting time of $D$ by the process 
$(Y_n(X_{n\wedge \nu_1}))_n$: 
$$
\nu_{1}=\inf\{ n>0 \,:\,
U_n< q(T_n,X_{n\wedge \nu_1},X_n)/K\}=
\inf\{ n>0 \,:\,
Y_n(X_{n\wedge \nu_1}) \in D\}
$$
$$
\nu_{2}=\inf\{ n>\nu_1 \,:\,
U_n< q(T_n,X_{\nu_1},X_{n\wedge \nu_1})/K\}=
\inf\{ n>\nu_1 \,:\,
Y_n(X_{n\wedge \nu_1}) \in D\}
$$
Since $(Y_n(X_{n\wedge \nu_1}))_n$
  is $\H$-adapted we conclude that $\nu_2$ is a stopping time for  $\H$.
Since $(X_n)_n$ is $\H$-adapted, the process $(X_{n\wedge \nu_2})_n$ is also $\H$-adapted. Iterating this argument we can show that all the random times $\nu_k$ are stopping times for  $\H$ and Step I is proved.

 \medskip

{\it Step II: for every $k\ge 0$, 
 $T_{\nu_k}$ is a stopping time for the filtration 
$\F^{\bar N,1}$}.

This is trivial for $k=0$, so assume $k\ge 1$. We write
$$
\{T_{\nu_k}\le t\}=\bigcup_{n\ge 0}
 \{\nu_k=n, T_{n}\le t\}$$
and we recall that $T_n$ is a stopping time for $\F^{\bar N,1}$ and, by Step I, that 
$\{\nu_k=n\}\in 
\calf^{\bar N,1}_{T_n}
$. It follows that $\{\nu_k=n, T_{n}\le t\}\in  \F^{\bar N,1}_t$ (by the very definition of $\calf^{\bar N,1}_{T_n}$) and therefore also 
$\{T_{n_k}\le t\}\in \F^{\bar N,1}_t$.

\medskip 

{\it Step III: for every $k\ge 0$, $X_{\nu_k}\,1_{\nu_k<\infty}$ is
 $\F^{\bar N,1}_{T_{\nu_k}}$-measurable}.

This is clear for $k=0$, since $\nu_0=0$, $T_0=0$ and $X_0$ is $\F^{\bar N}_0=\sigma(X_0)$-measurable. Next we  assume $k\ge1$. 

For any $B\subset S$ and any $t\ge0$ we have
$$
\{X_{\nu_k}\,1_{\nu_k<\infty}\in B, T_{\nu_k}\le t\}= \bigcup_{n\ge 1}
\{\nu_k=n, X_{n}\in B, T_{n}\le t\}.
$$
From Step I we have $\{\nu_k=n\}\in \F^{\bar N,1}_{T_{n}}$. Since $X_n$ is measurable with respect to  $  \F^{\bar N}_{T_{n}} \subset\F^{\bar N,1}_{T_{n}}$ it follows that 
$\{\nu_k=n, X_{n}\in B \}\in \F^{\bar N,1}_{T_{n}}$ and so
$\{\nu_k=n, X_{n}\in B, T_{n}\le t\}\in \F^{\bar N,1}_t$ (by definition of $\F^{\bar N,1}_{T_{n}}$) and finally we obtain
$\{X_{\nu_k}\,1_{\nu_k<\infty}\in B, T_{\nu_k}\le t\}\in \F^{\bar N,1}_t$ which proves Step III.

\medskip

Now adaptedness of $X^q$ follows from Steps II and III and the representation \eqref{rapprXqgen}.
\qed

We are now ready for the main result of this section.

\begin{Theorem}\label{compensatorecatenagen}
Suppose that Assumption \ref{settingprobabgen} holds and 
that  $\F^1=(\calf^1_t)_{t\ge0}$ is a filtration  in $(\Omega,\calf)$, with $\calf^1_\infty$ independent from the random elements in Assumption \ref{settingprobabgen}. With the previous notation,
 let  $q:\Omega\times[0,\infty)\times S\times S\to \R$ be
 $\calp(\F^{1})\otimes \calb(S)\otimes \calb(S)$-measurable and
 satisfy \eqref{qbdd}. Let us take the constant in Assumption \ref{settingprobabgen}-${\it 1}$ so large that  $K\ge C_q$.

Then the $\F^{ N,1}$- compensator  of $N(dt,dy)$ is
$$
\nu(dt,dy)= q(t,X^q_{t-},y)\,\mu(dy)\,dt.
$$
\end{Theorem}

\noindent {\bf Proof.}
First we note that $(X^q_{t-})$ is $\F^N$-predictable, so that by the measurability assumptions on $q$ the random measure $q(t,X^q_{t-},y)\,\mu(dy)\,dt$ is 
$\F^{ N,1}$-predictable.

Let $H(t,y)\ge 0$ be an 
$\F^{ N,1}$-predictable process. We have
$$
\E\int_S\int_0^\infty H(t,y)\,N(dt,dy)=\E\sum_{k\ge 1} H(T_{\nu_k},X_{\nu_k})\,1_{T_{\nu_k}<\infty}.
$$
For $n\ge 1$ and $k\ge 1$ such that  $\nu_{k-1}<n<\nu_{k}$  the inequality $U_n$ $\ge $  $q(T_n,X_{\nu_k},U_n)/K$
takes place. So we may rewrite the previous sum adding several null terms as follows:
$$
\sum_{k\ge 1} H_{T_{\nu_k}}(X_{\nu_k})\,1_{T_{\nu_k}<\infty}=
\\
\sum_{k\ge 1} 
\bigg[
\sum_{n= 1+\nu_{k-1}}^{\nu_k}
H(T_{n},X_{n})\,
1\Big(U_n< q(T_n,X_{\nu_{k-1}},X_n)/K\Big)
    \bigg]
    1_{T_{\nu_{k}}<\infty}
$$
(in each sum in square brackets  only the last term may be non-zero). Next note that, for  $k\ge 1$,
\begin{align*}
       X_{\nu_{k-1}}=X^q(t-) \hbox{ for } T_{\nu_{k-1}}<t\le T_{\nu_k} \quad \Longrightarrow \quad
    X_{\nu_{k-1}}=X^q(T_n-) \hbox{ for } \nu_{k-1}<n \le \nu_k.
\end{align*}   
So we obtain
\begin{align*}
\sum_{k\ge 1} H_{T_{\nu_k}}(X_{\nu_k})\,1_{T_{\nu_k}<\infty}
&=
\sum_{k\ge 1} 
\bigg[
\sum_{n= 1+\nu_{k-1}}^{\nu_k}
H(T_{n},X_{n})\,
1\Big(U_n< q(T_n,X^q(T_n-),X_n)/K\Big)
    \bigg]
    1_{T_{\nu_{k}}<\infty}
\\  &
=\sum_{n\ge 1} H(T_{n},X_{n})\,
 1\Big(U_{n}< q(T_{n},X^q(T_{n}-),X_{n})/{K}\Big).
\end{align*}
This may be written as an integral with respect to the random measure $\bar N$, leading to
$$
\E\int_S\int_0^\infty H(t,y)\,N(dt,dy)=
\E\int_S\int_0^\infty \int_0^1 H(t,y)
\,
 1\Big(u< q(t,X^q(t-),y)/{K}\Big) \,\bar N(dt,dy,du).
$$
From Lemma \ref{Xqpredictablegen} it follows that $(X^q(t-))$ is $\F^{\bar N,1}$-predictable and so is the integrand in the right-hand side of the last displayed formula. Recalling the form of the compensator of $\bar N$ we obtain
$$
\E\int_S\int_0^\infty H(t,y)\,N(dt,dy)=
\E\int_S\int_0^\infty \int_0^1 H(t,y)
\,
 1\Big(u< q(t,X^q(t-),y)/{K}\Big)\,du \, {K} \, dt\,\mu(dy).
$$
Noting that $q(t,X^q(t-),j)/{K}\le C_q/K\le 1$ 
we can compute the integral  in $du$ and 
we conclude that
$$
\E\int_S\int_0^\infty H(t,y)\,N(dt,dy)=
\E\int_S\int_0^\infty H(t,y)
\, q(t,X^q(t-),y)  \, dt\,\mu(dy)
=
\E\int_S\int_0^\infty H(t,y)
\, \nu(  dy,dt).
\qed
$$

\section{The partially observed control problem and its reformulations}
\label{secformulazionepb}

In this section we 
suppose that Assumption \ref{settingprobabgen}  holds. 
From now on we also assume that the state space $S$ is finite; we will use letters $i,j$ to denote its elements. We need to introduce a space $A$ of control actions where the  control process $(\alpha_t)$ takes values.
We also need to introduce  controlled transition rates $q(a,t,i,j)$, a function $h(t,i,a)$ to model the observation and real functions $f(t,i,a)$ and $g(i)$ to define the reward to be maximized; they may depend on the control action $a\in A$.  
According to the usual approach (see e.g. \cite{Bebook})
we will initially set the control problem under the reference probability measure $\P$, so that in particular the observation $W$ will be a given Brownian motion under $\P$.
This has the advantage that the corresponding filtration does not depend on the control process. 
 Here are the hypotheses we need and which be valid in the rest of the paper (in addition to Assumption \ref{settingprobabgen}).

\begin{Assumption}\label{ipcoeffpb}
    \begin{enumerate}
        \item $(W_t)_{t\ge0}$ is a standard $d$-dimensional Brownian motion defined in $(\Omega,\calf,\P)$; we denote $\F^W=(\calf^W_t)_{t\ge 0}$ its completed filtration.
     
        \item $S$ is a finite set with cardinality $N$. $A$ is a Polish space. $T>0$ and $\beta>0$ are given constants.
\item For every $i,j\in S$ ($i\neq j$) we are given numbers $g(i)\in\R$ and  functions
$$
q(\cdot,\cdot,i,j):A\times [0,\infty) \to [0,\infty), \quad
h(i,\cdot,\cdot):A\times [0,\infty) \to \R^d, \quad
f(i,\cdot,\cdot):A\times [0,\infty) \to \R,
$$
which        are Borel measurable, and there exists a constant $K_0$ such that
\begin{equation}
    \label{coeffbdd}
|q(a,t,i,j)| + |h(i,a,t)| + |f(i,a,t)|
+|g(i)|
 \le K_0,
\qquad a\in A;\,t\ge 0;\, i,j\in S\, (i\neq j).
\end{equation}
        
    \item 
     The constant in Assumption \ref{settingprobabgen}-${\it 1}$ is taken so large that
    $$
N\cdot q(a,t,i,j)\le K,
\qquad a\in A;\,t\ge 0;\,i,j\in S\,(i\neq j).
$$
\item    The random variables in Assumption \ref{settingprobabgen}-${\it 2}$ are uniformly distributed on $S$.
    \end{enumerate}
\end{Assumption}  
We complete the definition of the   rate matrix setting as usual
$$q(a,t,i,i)=-\sum_{j\neq i} q(a,t,i,j).
$$
We finally define the set of admissible controls of the partial observation problem as
$$
\cala=\{ \alpha:\Omega\times [0,\infty)\to A,\; \F^W\hbox{-predictable}\}.
$$

\subsection{The partially observed control problem for the reference probability}
\label{subsecrefprob}

For every $\alpha\in\cala$ 
we next define a corresponding controlled $S$-valued process  using the construction of the previous section. 
Instead of a general filtration $\F^1$ now we take the filtration 
$\F^W$. Then
we  consider the $\F^W$-predictable processes $(N\cdot q(\alpha_t,t,i,j))_t$
and we construct the corresponding process $X^q$ as in the previous section, that will now be called  $X^\alpha$. 
Explicitly, we define  $\nu_0=0$ and, for $k\ge 0$,
$$
\nu_{k+1}=\inf\{ n>\nu_k \,:\,
U_n< N\cdot q(\alpha_{T_n},T_n,X_{\nu_k},X_n)/K\},
$$
with the convention $\inf \emptyset =\infty$.
We take an element $\delta\notin S$, we set 
$T_{\nu_n}=\infty$ and $X_{\nu_n}=\delta$ if ${\nu_n}=\infty$ and
we 
consider the marked point process  $(T_{\nu_n},X_{\nu_n})_{n\ge 1}$. The corresponding 
$S\cup\{\delta\}$-valued process $(X^\alpha_t)_{t\ge 0}$ (starting from $X_0$ at time $0
$) defined by
$$ 
X^\alpha_t= X_{\nu_n},\quad \text{for}\quad  
T_{\nu_n}\le t< T_{\nu_{n+1}},\,n\ge0,
$$
is the controlled process corresponding to $\alpha\in\cala$.
On the finite state space $S$, measures $\mu(dy)$ are identified with their masses $\mu(j)$ at any point $j\in S$. For instance, the uniform distribution of the variables $X_n$ is $\mu(j)=1/N$ (which accounts for the factor $N$ in some of the previous formulae). Correspondingly, the random measure on $(0,\infty)\times S$  associated to  $(T_{\nu_n},X_{\nu_n})_{n\ge 1}$ is now denoted
$$
  N^\alpha(dt,j )=\sum_{n\ge 1} \delta_{(T_{\nu_n},X_{\nu_n})}(dt,j)\, 1_{T_{\nu_n}<\infty}.
$$
We denote $\F^{ N^\alpha,W}$  the filtration generated by $  N^\alpha$, $X_0$ and $W$. By Theorem \ref{compensatorecatenagen}, the 
  $\F^{ N^\alpha,W}$-compensator  of $N^\alpha(dt,j)$ is
$$
\nu^\alpha(dt,j)= q(\alpha_t,t,X^\alpha_{t-},j)\,dt.
$$ 
This formula justifies the interpretation of $X^\alpha$ as a Markov chain with ``stochastic transition rates'' given by $q(\alpha_t,t,i,j)$.

Having constructed the controlled processes $X^\alpha$ we can formulate the optimization problem by introducing the reward functional to be maximized. 
Let us define
$$
Z^\alpha_t=\exp\left(\int_0^t h(X^\alpha_s,\alpha_s,s)\,dW_s-\frac{1}{2} \int_0^t |h(X^\alpha_s,\alpha_s,s)|^2\,ds\right),
\qquad t\ge 0.
$$
The optimal control problem for a finite  horizon $T$ consists in maximizing the reward functional
$$
J_T(\alpha)=\E\left[ \int_0^T Z^\alpha_t\,
f(X^\alpha_t,\alpha_t,t)\,dt+ Z^\alpha_T\,
g(X^\alpha_T)
\right]
$$
over all $\alpha\in\cala$.
The infinite horizon optimal control problem consists in maximizing the discounted reward functional
$$
J_\infty(\alpha)=\E\left[ \int_0^\infty 
e^{-\beta t} Z^\alpha_t\,
f(X^\alpha_t,\alpha_t)\,dt
\right]
$$
 with discount rate $\beta$. In the infinite horizon case the functions $q$, $h$ and $f$ are taken to be time-independent. The occurrence of the process $Z^\alpha$ is explained in the following reformulation.

\subsection{The partially observed control problem for the physical probability}

Here we show that the previous formulation corresponds to the original control problem outlined in the introduction, provided an appropriate weak sense formulation is given. The first step will be to construct a physical probability under which the observation has the desired form 
\eqref{obsintro}. 
Let us start with some preliminary remarks.

We note that, since $h$ is bounded, the process 
$Z^\alpha$ introduced above is 
a continuous $\F^{ N^\alpha,W}$-martingale.
By the form of the compensator, the processes
$$
M^{j,\alpha}_t:=
N^\alpha((0,t],j)- \nu^\alpha((0,t],j), \qquad t\ge0,
$$
are also $\F^{ N^\alpha,W}$-martingales.
 Since they are locally of integrable variation, they are purely discontinuous martingales, hence orthogonal to $Z^\alpha$, which means that the products $
M^{j,\alpha}\,Z^\alpha$ are local martingales.

Now let us define, for each $t\ge0$, a consistent family of   probabilities $\bar\P^\alpha_t$
on $\calf^{ N^\alpha,W}_t$
corresponding to the Doléans exponential $Z^\alpha$, namely:
$$
d\bar\P^\alpha_t=Z^\alpha_t\,d\P\Big|_{\F^{ N^\alpha,W}_t},
$$
as well as the process
$
B^\alpha_t:= W_t-\int_0^th(X_s^\alpha,\alpha_s,s)\,ds$, $t\ge0$.
Then the following holds.
\begin{enumerate}
    \item For every $T>0$, under $\bar\P^\alpha_T$ the process $B^\alpha$ is a Wiener process on $[0,T]$:

    This follows from the Girsanov theorem.

    \item  For every $T>0$, under $\bar\P^\alpha_T$ the random measure $N^\alpha(dt,j)$ has the same $\F^{ N^\alpha,W}$-compensator $\nu^\alpha(dt,j)= q(\alpha_t,X^\alpha_{t-},j)\,dt.$ 

Indeed, the processes 
$M^{j,\alpha}$ remain $ \F^{N^\alpha,W}$-martingales under $\bar\P^\alpha$, because $M^{j,\alpha}Z^\alpha$ are $\F^{N^\alpha,W}$-local martingales under $\P$.
\end{enumerate}

It is easy to check that the reward functionals can be written
$$
J_T(\alpha)=\bar\E^\alpha_T\left[ \int_0^T 
f(X^\alpha_t,\alpha_t,t)\,dt  + 
g(X^\alpha_T)
\right],
$$
where $\bar\E^\alpha_T$ denotes expectation under $\bar\P^\alpha_T$ and, in the infinite horizon case,
\begin{equation}
    \label{Jinftybar}
J_\infty(\alpha)=\lim_{T\to\infty}\bar\E^\alpha_T\left[ \int_0^T
e^{-\beta t} 
f(X^\alpha_t,\alpha_t)\,dt
\right].
\end{equation}
This is the original optimal control problem outlined in the introduction: indeed, on each interval $[0,T]$, the controlled Markov chain $X^\alpha$ has 
$\F^{ N^\alpha,W}$-compensator  
$$
\nu^\alpha(dt,j)= q(\alpha_t,t,X^\alpha_{t-},j)\,dt.
$$ 
and the observation process has the form
$$
W_t=\int_0^th(X_s^\alpha,\alpha_s,s)\,ds+
B^\alpha_t,\qquad  t\ge0.
$$
This   optimization problem is in  weak form, since the physical probabilities $\bar\P_T^\alpha$ and the observation noise $B^\alpha$ depend on $\alpha$.

\begin{Remark}\emph{Suppose that, in the finite horizon case, one wishes to maximize a functional of the form
$$ 
\bar\E_T^\alpha \left[ \sum_j
\int_0^T
\ell(X^\alpha_{t-},j,\alpha_t,t)\,
N^\alpha(dt, j)
\right]
$$
for some bounded Borel measurable real function $\ell(i,j,a,t)$. This is a running reward depending explicitly on the random measure $N^\alpha(dt,j)$. Since the previous integrand is 
$\F^{ N^\alpha,W}$-predictable, this is the same as
$$ 
\bar\E_T^\alpha\left[ \sum_j
\int_0^T 
\ell(X^\alpha_t,j,\alpha_t,t)\,
q(\alpha_t, t,X^\alpha_t,j)\,
dt  
\right]
$$
which has the form of the running reward considered before, setting
$f(i,a,t)=\sum_j\ell(i,j,a,t) q(a,t,i,j)$. Similar considerations apply to the infinite horizon case.
}
\end{Remark}

\subsection{The separated optimal control problem}

Still assuming that Assumptions \ref{settingprobabgen} and \ref{ipcoeffpb} hold true, we come back to the optimization problem formulated in  subsection \ref{subsecrefprob}, that we rewrite in a different equivalent form.
It is convenient to introduce the 
generator   $Q_t^a$ of the controlled, time-dependent Markov chain, which maps any function $\phi:S\to \R$ to the function 
$Q_t^a\phi:S\to \R$
given by
$$
Q^a_t\phi (i)=\sum_j\phi(j) \, q(a,t,i,j)
\qquad i\in S.
$$
Next we introduce 
the unnormalized conditional law setting, for any $\phi$, 
$$
\rho_t(\phi)=\E[\phi(X_t^\alpha)\,Z^\alpha_t\,|\, \calf_t^W].
$$
The conditional expectation is taken under 
 the reference probability $\P$. The process $(\rho_t(\phi))_t$ is understood  as the optional projection of $(\phi(X_t^\alpha)\,Z^\alpha_t)$ for the filtration $\F^W$ and the formula defines an optional process with values in the space of nonnegative finite measures over $S$; we refer to \cite{BaCribook} for details. It is easy to show that the reward functionals can be written
 $$
J_T(\alpha)=\E\left[ \int_0^T  
\rho_t(f(\cdot,\alpha_t,t))\,dt + \rho_T(
g)
\right]\quad \text{and}\quad
J_\infty(\alpha)=\E\left[ \int_0^\infty e^{-\beta t} 
\rho_t(f(\cdot,\alpha_t))\,dt
\right].
$$
 The motivation to introduce the process $\rho$ is the well known fact that it is a solution to the 
 the Zakai filtering equation which, in the present case of a finite-state Markov chain, is also called  the Wonham filter:
for every $\phi:S\to \R$, 
$$
d\rho_t(\phi)= \rho_t(Q_t^{\alpha_t}\phi)\,dt
+ \rho_t\Big(h(\cdot,\alpha_t,t)\phi(\cdot)\Big)\,dW_t, \qquad \rho_0(\phi)=\E[\phi(X_0)].
$$
We will soon see that for every admissible control and any initial condition $\rho_0$ there exists a unique solution.
This way we have obtained the so-called separated problem: the state equation is the controlled Wonham filter and the reward functionals depend on the control and the corresponding state trajectory.

\subsection{Optimal control  of the Wonham filter: setting and preliminary results}

Here we introduce the appropriate formulation for our optimization problem, that will be the aim of the analysis in all the following sections.

We start from the separated problem and we first note that the 
space of nonnegative finite measures over $S$ can be identified with
$$
 D=\{x=(x_1,\ldots,x_N)\in \R^N\,:\, x_i\ge0,\, i=1,\ldots,N\}.
 $$ 
 Let us define
$\rho_t^i:=\rho_t(1_{\{i\}})$, where $1_{\{i\}}:S\to\R$ is the indicator function of  state $i$.  Noting that for every $\phi$ we have $\rho_t(\phi)=\sum_i \rho_t^i\,\phi(i)$, easy computations show  the controlled Wonham filtering equations can be rewritten as a system of SDEs for the process $(\rho_t^1,\ldots,\rho_t^N)$.
Allowing a general starting time $t\in [0,T]$ and a general initial condition $x=(x_i)\in D$, the finite horizon problem is 
\begin{equation}
    \label{pfHf}
\left\{
\begin{array}{rclr}
d\rho_s^i&=&\displaystyle \sum_j \rho_s^j \,q(\alpha_s,s,j,i)\,ds
+ \rho_s^i\,h(i,\alpha_s,s)\,dW_s,
& s\in[ t,T],\, i\in S,
\\
\rho_t^i&=&x_i,& x\in D,\,  i\in S,
\\
J_T(t,x,\alpha)&=&\displaystyle\E\left[ \int_t^T  \sum_i
\rho_s^i\,f(i,\alpha_s,s)\,ds + \sum_i\rho_T^i\,g(i)
\right],&
\end{array}\right.
\end{equation}
and the infinite horizon problem starting at time $0$ is
\begin{equation}
    \label{pfHinf}
\left\{
\begin{array}{rclr}
d\rho_t^i&=&\displaystyle \sum_j \rho_s^j \,q(\alpha_t,j,i)\,dt
+ \rho_s^i\,h(i,\alpha_t)\,dW_t,
&  t\ge 0,\,i\in S,
\\
\rho_0^i&=&x_i,& x\in D,\,  i\in S,
\\
J_\infty(x,\alpha)&=&\displaystyle\E\left[ \int_0^\infty e^{-\beta t}  \sum_i
\rho_t^i\,f(i,\alpha_t)\,dt
\right].&
\end{array}\right.
\end{equation}
 Clearly, the problem starting at time $t\ge0$ also admits a rephrasing in the original formulations, even under the physical probability.
Now we define the value functions for these problems:
\begin{equation}
    \label{valuefunctions}
    V(t,x)= \sup_{\alpha\in \cala} J_T(t,x,\alpha),\qquad
    V(x)= \sup_{\alpha\in \cala} J_\infty(x,\alpha),\qquad \quad t\in [0,T],\, x\in D.
\end{equation} 

In the following Proposition we collect some preliminary properties of these optimization problems and the corresponding value functions.

\begin{Proposition}\label{investimeV}
Suppose that Assumptions \ref{settingprobabgen} and \ref{ipcoeffpb} hold true, and that the coefficients do not depend on time in the infinite horizon case.
\begin{enumerate}
    \item For the solution to \eqref{pfHf}
we have, $\P$-a.s., $\rho_s^i\in D $ for every $s\in [t,T]$, $i\in S$.
If $x_i> 0$ for every $i\in S$ then, $\P$-a.s.,  $\rho_s^i> 0$ for every $s\in [t,T]$ and $i\in S$.
Similar results hold for the solution to \eqref{pfHinf}. 

    \item For every $t\in [0,T]$ the function $x\mapsto V(t,x)$ is convex; moreover there exists a constant $C$ such that for every $t\in [0,T]$, $x,\bar x\in D$, 
    \begin{equation}
        \label{Vtlinearlip}
|V(t,x)-V(t,\bar x)|\le C\,|x-\bar x| ,\qquad
|    V(t,x)|\le C\,(1+|x|).
\end{equation}
\item the function $x\mapsto V(x)$ is convex, hence locally Lipschitz; moreover for every  $x\in D$, 
    \begin{equation}
        \label{Vbdd}
|    V(x)|\le \frac{1}{\beta}\,\sup |f|.
\end{equation}
\end{enumerate}

\end{Proposition}

\noindent
{\bf Proof.}
{\it 1.}  We note that the equation is linear with respect to the state variable and has bounded (stochastic) coefficients. So the classical conditions on Lipschitz continuity and linear growth hold true and the equation has a unique continuous $\F^W$-adapted solution starting from any $x\in \R^N$. If $x=0\in D$ then $\rho =0$. If $x\in D $ and $x\neq 0$ then $cx$ is a probability measure on $S$ for $c=(\sum_ix_i)^{-1}$ so that $c\rho_t$ coincides with the unnnormalized conditional distribution and so it is a nonnegative measure on $S$; it follows that $\rho^i_t\ge0$.

To prove the result of strict positivity we write
the state equations   as a system of ordinary (deterministic) differential equations with stochastic coefficient: this is the so-called {robust form} of the Zakai equation. We set
$$
\nu_t^i=\rho_t^i\,\exp\left( -\int_0^t h(i,\alpha_s)\,dW_s\right),
$$
and we look for the equation satisfied by $\nu^i$. Computing the Ito differential $d\nu_t^i$,  after some calculations we have
$$
 d\nu^i_t= 
\sum_{j} \nu_t^j \,q(\alpha_t,j,i)
\exp\left( \int_0^t [h(j,\alpha_s)-h(i,\alpha_s)]\,dW_s\right)dt-\frac{1}{2}
\nu_t^i \, |h(i,\alpha_s)|^2dt
.
$$
So we obtain the robust equation in the form
$
\frac{d}{dt}\nu^i_t= 
\sum_j a^{ij}_t\nu_t^j  
$ setting
$$
a^{ii}_t= q(\alpha_t,i,i)-\frac{1}{2}|h(i,\alpha_t)|^2,\qquad
a^{ij}_t= q(\alpha_t,j,i)
\exp\left( \int_0^t [h(j,\alpha_s)-h(i,\alpha_s)]\,dW_s\right), \qquad j\neq i.
$$
 We have
$$
\frac{d}{dt}\nu^i_t=  a^{ii}_t\nu^i_t+
\sum_{j\neq i} a^{ij}_t\nu_t^j  =
a^{ii}_t\nu^i_t+g^i_t
$$
where $g^i_t:= 
\sum_{j\neq i} a^{ij}_t\nu_t^j\ge 0$ by the nonnegativity result already proved. It follows that
$$
\nu_t^i=\nu_0^i \exp\left( \int_0^t  a^{ii}_s\,ds\right) +
\int_0^t \exp\left( \int_s^t  a^{ii}_r\,dr\right) g^i_s\,ds
\ge \nu_0^i \exp\left( \int_0^t  a^{ii}_s\,ds\right) >0
$$
  for every $i\in S$ if $\nu_0^i> 0$ for every $i\in S$. The same clearly holds for $\rho^i$.

  \medskip
{\it 2.} 
Let $\rho$, $\bar\rho$ denote the solutions starting from $x,\bar x$. By standard estimates we have $\E\sup_{s\in [t,T]}|\rho_s-\bar\rho_s|^2\le C|x-\bar x|^2$ for some constant $C$ (depending also on $T$). By the boundedness of $f$ and $g$
it follows easily that
$|J_T(t,x,\alpha)-
J_T(t,\bar x,\alpha)|^2\le C'|x-\bar x|^2$ for some constant $C'$ independent of $\alpha$. \eqref{Vtlinearlip} follows immediately.

We note that the state equation and the reward functional are linear. It follows that $x\mapsto J_T(t,x,\alpha)$ is linear and $x\mapsto V(t,x)$ is convex as the supremum of linear functions.

  \medskip

{\it 3.} 
From formula 
    \eqref{Jinftybar} it follows that
    $|    J_\infty(x,\alpha)|\le   \liminf_{T\to\infty}  \int_0^T
e^{-\beta t} \sup |f|\,dt
\le \sup |f|/\beta
$ and the estimate on $V$ holds. 
Convexity follows from linearity as before.
    \qed

\begin{Remark}
{\emph{
By similar arguments it is also easy to show that $x\mapsto V(x)$ is globally Lipschitz provided $\beta>0 $ is sufficiently large.
}}    
\end{Remark}

\section{Dynamic programming equation for infinite horizon: viscosity theory}
\label{sechjbellitt}

In this section we study the value function $V$ for the problem \eqref{pfHinf}. We suppose   that Assumptions \ref{settingprobabgen} and \ref{ipcoeffpb} hold  and that the coefficients do not depend on time. 

We will show that $V$ is the unique viscosity solution to the dynamic programming equation (also called Hamilton-Jacobi-Bellman equation - HJB) which takes the following form: for $x\in D$,
\begin{equation}
\label{HJBellittca}
\beta v(x)   -\sup_{a\in A}\bigg[
\frac{1}{2}\sum_{ij}\partial_{ij}^2v(x )\,
x_ix_j\sum_{k=1}^dh_k(i,a)h_k(j,a)
+\sum_{ij} \partial_iv(x)x_jq(a,j,i)
+ 
\sum_i
x_i\,f(i,a)  
\bigg]     =0.
\end{equation}
This will be written as follows: denoting $Dv$ and $D^2v$ the gradient and the Hessian matrix of $v$ we have
$$
\beta v +F(x,Dv,D^2v)=0,
$$
where, for $x\in D$, $p\in\R^N$ and $X\in \cals(N) $ (the set of symmetric real $N\times N$ matrices)
$$
F(x,p,X)= -\sup_{a\in A}\bigg[
\frac{1}{2}\hbox{Trace}\Big(\Sigma(x,a )\Sigma(x,a)^TX\Big) +\langle p, b(x,a)\rangle + f(x,a)
\bigg]  
$$
where the $N\times d$ matrix
$\Sigma(x,a )$, the vector $b(x,a)$
and the real function $f(x,a)$ are
$$
\Sigma(x,a )=
 (x_i\,h_k(i,a))_{ik},
 \quad
 b(x,a)=\Big( \sum_j x_jq(a,j,i)\Big)_i, 
 \quad
 f(x,a)= \sum_i
x_i\,f(i,a) ,
$$
for $i=1,\ldots,N$, $k=1,\dots,d $.
Recall that we assume (see \eqref{coeffbdd})
$$
|h_k(i,a)| + |q(a,i,j)| + |f(i,a)|
 \le K_0,
\qquad a\in A;\,k=1,\ldots,d;\, i,j=1,\ldots,N
$$
for some constant $K_0$. Therefore $\Sigma(x,a)$, $b(x,a)$, $f(x,a)$ are Lipschitz continuous in $x$, uniformly in $a$, with a Lipschitz constant that only depends on $K_0,N,d$.
An easy computation shows that 
\begin{equation}\label{FLip}
    |F(x,p,X)-F(x,q,Y)|\le C_0(|x|\,|p-q|+ |x|^2\,|X-Y|)
\end{equation}
for every $x\in D$, $p,q\in \R^N$, $X,Y\in \cals(N)$ and for some constant $C_0$  depending   only  on $K_0,N,d$. We also note
that $F$ is a continuous function of all its arguments.

Let us briefly recall the definition of viscosity sub-/supersolutions.  We find it convenient to write it using sub- and superjets. The equivalence with the other definition based on the use of test functions is well known and can be found e.g. in \cite{CrIsLi1992}, \cite{FlSobook}. 

For $u:D\to \R$ and $x\in D$, the superjet $J ^{2,+}  u( x)$ is the set of pairs $(p,X)\in \R^N\times \cals(N)$ such that
$$
u(y)\le u(x) + \langle p,y-x\rangle + \frac{1}{2}\langle X(y-x),y-x\rangle + o(|y-x|^2)
\text{  as  } y\in D, y\to x.
$$
The closure $\bar J ^{2,+}  u( x)$ consists of the pairs $(p,X)$ such that there exists a sequence $(x_n,p_n,X_n) \in D\times \R^d\times \cals(N)$ such that $x_n\to x$, $p_n\to p$, $X_n\to X$, $u(x_n)\to u(x)$, $(p_n,X_n)\in J ^{2,+}  u( x_n)$. We define subjets setting $J ^{2,-}  u( x)=-J ^{2,+}  (-u)( x)$,  $\bar J ^{2,-}  u( x)=-\bar J ^{2,+}  (-u)( x)$. 

We say that an upper semicontinuous function $u:D\to \R$ is a viscosity subsolution if for any $x\in D$
$$
(p,X)\in \bar J ^{2,+}  u( x)\quad \Longrightarrow\quad
\beta u(x)+F(x,p,X)\le 0. 
$$
A lower semicontinuous function  $v:D\to \R$ is called a viscosity supersolution if for any $x\in D$
$$
(p,X)\in \bar J ^{2,-}  v( x)\quad \Longrightarrow\quad
\beta v(x)+F(x,p,X)\ge 0. 
$$
Finally, a viscosity solution is both a sub- and supersolution.

The main result of this section is the following comparison result, which immediately implies uniqueness of bounded viscosity solution to the HJB equations and allows to prove the characterization result for the value function.

\begin{Theorem}
    Let $u$ be an upper semicontinuous  subsolution bounded above,   $v$ a lower semicontinuous supersolution bounded below. Then $u\le v$.
\end{Theorem}

{\bf Proof.}
Define, for $x,y\in D$, $\alpha>0$, $\delta>0$,
$$
\Phi(x,y)=u(x)-v(y)-\frac{\alpha}{2}|x-y|^2- \delta\log (\gamma + |x|^2)
- \delta\log (\gamma + |y|^2).
$$
Here $\alpha$ will eventually tend to $\infty$, $\delta$ to $0$ and $\gamma>0$ will be fixed later, sufficiently large.
By the boundedness assumption there exists a maximum point $(\hat x,\hat y)\in D\times D$. $\Phi,\hat x,\hat y$ depend on $\alpha, \delta,\gamma$ but we omit this dependence in the notation. By standard arguments (see e.g. \cite{CrIsLi1992} Lemma 3.1 or Proposition 3.7) for fixed $\delta,\gamma$ we have
\begin{equation}
    \label{alphaediffi}
\alpha|\hat x-\hat y|^2\to0 , \quad |\hat x-\hat y|\to0
\end{equation}
as $ \alpha\to\infty$. 
For every $x\in D$ we have
$$
\Phi(\hat x,\hat y)\ge \Phi(x,x)= u(x)-v(x) -2 \delta\log (\gamma + |x|^2),
$$
so, letting
$$\theta_\delta=
\sup_{x\in D} (u(x)-v(x) -2 \delta\log (\gamma + |x|^2)),
$$ 
we have $ \Phi(\hat x,\hat y)\ge \theta_\delta$, which implies in particular 
\begin{equation}
    \label{phiethetadelta}
    \theta_\delta + \delta\log (\gamma + |\hat x|^2)
+ \delta\log (\gamma + |\hat y|^2)
\le
u(\hat x)-v(\hat y).
\end{equation}
We note that $\theta_\delta$ is decreasing in $\delta>0$. We claim that $\lim_{\delta \to 0}\theta_\delta\le 0$, so that   for every $\delta$ sufficiently small we have 
$ u(x)-v(x) -2 \delta\log (\gamma + |x|^2)\le 0$ for every $x\in D$ and letting $\delta\to 0$ we obtain the desired conclusion $u\le v$. To prove the claim we will show    that assuming $\lim_{\delta \to 0}\theta_\delta\in (0,\infty]$ leads to a contradiction.

Let us define
$$
g(x)= \log (\gamma + |x|^2),
\qquad
\tilde u(x)= u(x)-\delta g(x),
\qquad
\tilde v(y)= v(y)+\delta g(y).
$$
Then $(\hat x,\hat y)$ is a maximum point of $ \tilde u(x)-\tilde v(y)  -\frac{\alpha}{2}|x-y|^2$. By the Crandall-Ishii Lemma (see \cite{CrIs1990}, or \cite{CrIsLi1992} Theorem 3.2 and the discussion that follows) there exist $X,Y\in\cals(N)$ such that
\begin{equation}
    \label{getti}
(\alpha(\hat x-\hat y),X)\in \bar J ^{2,+}\tilde u(\hat x),
\quad
(\alpha(\hat x-\hat y),Y)\in \bar J ^{2,-}\tilde v(\hat y),
\quad 
\left(
\begin{array}{cc}
    X & 0 \\
    0 & -Y
\end{array}\right)\le 
3\alpha\left(
\begin{array}{cc}
    I & -I \\
    -I & I
\end{array}\right).
\end{equation}
Since $g$ is smooth, it follows that
$$
\Big(\alpha(\hat x-\hat y)+\delta Dg(\hat x),X+\delta D^2g(\hat x)\Big)\in \bar J ^{2,+}  u(\hat x),
\quad
\Big(\alpha(\hat x-\hat y)-\delta Dg(\hat y),Y-\delta D^2g(\hat y)\Big)\in \bar J ^{2,-}  v(\hat y)
$$
and since $u,-v$ are subsolutions,
$$
\beta u(\hat x)+F\Big(\hat x,\alpha(\hat x-\hat y)+\delta Dg(\hat x),X+\delta D^2g(\hat x)\Big)\le 0,
\quad
\beta v(\hat y)+F
\Big(\hat y,\alpha(\hat x-\hat y)-\delta Dg(\hat y),Y-\delta D^2g(\hat y)\Big)\ge 0.
$$
Subtracting these inequalities and recalling 
 \eqref{phiethetadelta} we obtain
 \begin{align*}
     &
   \beta \theta_\delta +\beta \delta\log (\gamma + |\hat x|^2)
+\beta \delta\log (\gamma + |\hat y|^2)
\\&
\le F
\Big(\hat y,\alpha(\hat x-\hat y)-\delta Dg(\hat y),Y-\delta D^2g(\hat y)\Big)
-
F\Big(\hat x,\alpha(\hat x-\hat y)+\delta Dg(\hat x),X+\delta D^2g(\hat x)\Big).
 \end{align*}
 Using \eqref{FLip}, the right-hand side can be estimated by
 $$ 
F
\Big(\hat y,\alpha(\hat x-\hat y) ,Y \Big)
-
F\Big(\hat x,\alpha(\hat x-\hat y) ,X \Big)+ 
\delta   C_0( |\hat y|\,| Dg(\hat y)|+ |\hat y|^2\,|D^2g(\hat y)|) 
+
\delta   C_0( |\hat x|\,| Dg(\hat x)|+ |\hat x|^2\,|D^2g(\hat x)|) .
$$
By explicit computations,
$$
Dg( x)= \frac{2x}{\gamma + |x|^2},
\qquad
D^2g( x)= \frac{2I}{\gamma + |x|^2}-
\frac{4x\otimes x}{(\gamma + |x|^2)^2},
$$
so that
$$
 C_0( | x|\,| Dg( x)|+ |x|^2\,|D^2g( x)|)\le
 C_1
$$
where $C_1$ is another constant    depending   only  on $K_0,N,d$ (and not on $\gamma>0$). It follows that
$$
 \beta \theta_\delta +\beta \delta\log (\gamma + |\hat x|^2)
+\beta \delta\log (\gamma + |\hat y|^2)\le
F
\Big(\hat y,\alpha(\hat x-\hat y) ,Y \Big)
-
F\Big(\hat x,\alpha(\hat x-\hat y) ,X \Big)+ 2
\delta C_1.
$$
Choosing $\gamma>0$ so large that $\beta\log\gamma \le C_1$ we arrive at
$$
 \beta \theta_\delta  \le
F
\Big(\hat y,\alpha(\hat x-\hat y) ,Y \Big)
-
F\Big(\hat x,\alpha(\hat x-\hat y) ,X \Big) .
$$
It is well known (see \cite{CrIsLi1992} Example 3.6) that the right-hand side can be estimated as follows:
$$
 \beta \theta_\delta  \le \omega(
 \alpha|\hat x-\hat y|^2 + |\hat x-\hat y|)
$$
for a modulus $\omega$ (i.e. a function $\omega:[0,\infty)\to [0,\infty)$ such that $\omega(0+)=0$) that only depends on the Lipschitz constants of  $\Sigma $, $b $, $f $, hence only  on $K_0,N,d$. Letting $\alpha\to \infty$ and recalling \eqref{alphaediffi} we have then
$\beta \theta_\delta  \le 0 $ which leads to a contradiction with the assumption that $\beta>0$ and $\lim_{\delta\to 0}\theta_\delta>0$.
\qed

This was the main step to the following result that summarizes the main conclusions on the control problem

\begin{Theorem}
Suppose   that Assumptions \ref{settingprobabgen} and \ref{ipcoeffpb} hold  and that the coefficients do not depend on time. Then the   value function $V$ for the problem \eqref{pfHinf} is the unique bounded viscosity solution of the HJB equation \eqref{HJBellittca}.
\end{Theorem}

{\bf Proof.}  Boundedness of $V$ was proved in \eqref{Vbdd} and  uniqueness follows from the previous result. Under our assumptions it is well known that $V$ satisfies a dynamic programming principle and that it is a viscosity solution: see e.g. 
\cite{Phbook09}
or 
\cite{FlSobook}.
\qed

\section{Dynamic programming equation for finite horizon: viscosity theory}
\label{sechjbparab}

In this section we study the value function $V$ for the problem \eqref{pfHf}. We still suppose that  Assumptions \ref{settingprobabgen} and \ref{ipcoeffpb} hold. 
We will show that $V$ is the unique viscosity solution to the  
 HJB equation. Here this is an equation for a function
$v(t,x)=v(t,x_1,\ldots,x_N)$ on the domain $(0,T)\times D$ of the form
\begin{equation}
    \label{HJBparabolica}
- v_t(t,x)   -\sup_{a\in A}\bigg[
\frac{1}{2}\sum_{ij}\partial_{ij}^2v(t,x )\,
x_ix_j\sum_{k=1}^dh_k(i,a,t)h_k(j,a,t)
+\sum_{ij} \partial_iv(t,x)x_jq(a,t,j,i)
+ 
\sum_i
x_i\,f(i,a,t) 
\bigg]     =0,
\end{equation}
with the boundary condition
\begin{equation}\label{HJBparabolicaterminal}
 v(T,x)= \sum_i
x_i\,g(i), \qquad x\in D.   
\end{equation}
 Denoting $Dv$ and $D^2v$ the gradient and the Hessian matrix of $v$ with respect to $x$, we have
$$
-v_t +F(t,x,Dv,D^2v)=0,
$$
where $F$ is defined 
for $t\in [0,T]$, $x\in D$, $p\in\R^N$ and $X\in \cals(N) $  by
$$
F(t,x,p,X)= -\sup_{a\in A}\bigg[
\frac{1}{2}\hbox{Trace}\Big(\Sigma(x,a,t )\Sigma(x,a,t)^TX\Big) +\langle p, b(x,a,t)\rangle + f(x,a,t)
\bigg]  
$$
$\Sigma(x,a,t )$, $b(x,a,t)$
and   $f(x,a,t)$
are defined similarly as before,
but possibly depending on $t$.
The inequality 
\eqref{FLip} still holds for every $t$, with the same constant $C_0$. We will assume that $h,q,f$ are continuous in $t\in [0,T]$ uniformly in $a\in A$, so that $F$ is a continuous function of all its arguments.

We report the standard definitions of viscosity sub- and supersolution using parabolic sub/superjets. 

For $u:(0,T)\times D\to \R$, $t\in (0,T)$, $x\in D$, the parabolic superjet $P ^{2,+}  u( t,x)$ is the set of triples $(a,p,X)\in \R\times \R^N\times \cals(N)$ such that
$$
u(s,y)\le  u(t,x) + a(s-t)+\langle p,y-x\rangle + \frac{1}{2}\langle X(y-x),y-x\rangle + o(|y-x|^2+|s-t|)
$$
 as  $ y\in D, y\to x$,  $s\in (0,T),s\to t$.
The closure $\bar P ^{2,+}  u(t, x)$ consists of the triples $(a,p,X)$ such that there exists a sequence $(a_n,x_n,p_n,X_n) \in (0,T)\times D\times \R^d\times \cals(N)$ such that $a_n\to a,$ $x_n\to x$, $p_n\to p$, $X_n\to X$, $u(t_n ,x_n)\to u(t ,x)$, 
$(a_n,p_n,X_n)\in J ^{2,+}  u(t_n ,x_n)$. We define subjets setting $P ^{2,-}  u( t,x)=-P ^{2,+}  (-u)( t,x)$,  $\bar P ^{2,-}  u(t, x)=-\bar P ^{2,+}  (-u)( t,x)$. 

We say that an upper semicontinuous function $u:(0,T]\times D\to \R$ is a viscosity subsolution if for any $t\in (0,T)$, $x\in D$
$$
(a,p,X)\in \bar P ^{2,+}  u(t, x)\quad \Longrightarrow\quad
a+F(t,x,p,X)\le 0,
$$
and moreover $u(T,x)\le g(x)$ for $x\in D$.
A lower semicontinuous function  $v:(0,T]\times D\to \R$ is called a viscosity supersolution if for any $t\in (0,T]$, $x\in D$
$$
(a,p,X)\in \bar P ^{2,-}  v( t,x)\quad \Longrightarrow\quad
a+F(t,x,p,X)\ge 0, 
$$
and moreover $v(T,x)\ge g(x)$ for $x\in D$.
Finally, a viscosity solution is both a sub- and supersolution.

We first prove the following comparison result.

\begin{Theorem} Suppose that  $h,q,f$ are continuous in $t\in [0,T]$ uniformly in $a\in A$.
    Suppose that $u,v:(0,T]\times D\to \R$ are  upper and lower semicontinuous, respectively. Let u be a  subsolution and   $v$ a  supersolution  satisfying
\begin{equation}
    \label{uvboundary}
    u(T,x)\le  v(T,x)
    \qquad x\in D.
\end{equation}  
    Suppose moreover that there exists a constant $C_1>0$ such that
\begin{equation}
    \label{uvlineargrowth}
    u(t,x)\le C_1(1+|x|),
    \; \;v(t,x)\ge -C_1(1+|x|),
    \qquad 
    t\in (0,T], x\in D.
\end{equation}
     Then $u\le v$ on $(0,T]\times D$.
\end{Theorem}

{\bf Proof.} {\it Step I.} We will first prove this result for sub/supersolutions to the equation
\begin{align}
    \label{HJBparabmodif}
-v_t+K\,v(t,x) +F(t,x,Dv,D^2v)=0,
\end{align}
where $K>0$ will be taken sufficiently large (in fact, satisfying $K\ge 2C_0$, compare \eqref{FLip}).  The general case will then be reduced to this one.

For $\delta\in (0,1]$ define
$$\theta_\delta=
\sup_{x\in D,0<t\le T} \left[u(t,x)-v(t,x) -2 \delta  (1 + |x|^2)-\frac{\delta}{t}\right].
$$ 
By \eqref{uvlineargrowth} the function in square parenthesis does not exceed 
$2 C_1 (1 + |x|)-2 \delta  (1 + |x|^2)$, so that $\theta_\delta<\infty$. $\theta_\delta$ is decreasing in $\delta$. If $\lim_{\delta\to 0}\theta_\delta\le 0$ then 
$$
u(t,x)-v(t,x) -2 \delta  (1 + |x|^2)-\frac{\delta}{t} \le 0,
\qquad 0<t\le T,x\in D, 0<\delta\le 1
$$
and letting $\delta\to 0$ we have
$u(t,x)\le v(t,x)$ and the conclusion is reached. 

Assume on the contrary that  $\lim_{\delta\to 0}\theta_\delta\in (0,\infty]$: we will see that this leads to a contradiction. 
Take $\bar\theta>0$ and $\delta >0$ such that $\theta_\delta\ge \bar\theta$. From now on we fix   $\delta $ and omit to indicate that  several quantities in the sequel may depend on it.

Define, for $x,y\in D$, $t\in (0,T]$, $\alpha>0$,
$$
\Phi_\alpha (t,x,y)=u(t,x)-v(t,y)- \delta  (1 + |x|^2)
- \delta  (1 + |y|^2)-\frac{\delta}{t}
-\frac{\alpha}{2}|x-y|^2.
$$
Later we will let  $\alpha\to\infty$.
From \eqref{uvlineargrowth} it follows that 
$$
\Phi_\alpha (t,x,y)\le 
C_1(1+|x|)
- \delta  (1 + |x|^2)
+C_1(1+|y|)- \delta  (1 + |y|^2)-\frac{\delta}{t}
$$
and since $\Phi_\alpha$ is upper semicontinuous it achieves a maximum at a point $(t_\alpha,x_\alpha,y_\alpha)\in (0,T]\times D\times D$.
Since
$$
\Phi_\alpha (t_\alpha,x_\alpha,y_\alpha)\ge
\Phi_\alpha (t,x,x)=
u(t,x)-v(t,x) -2 \delta  (1 + |x|^2)-\frac{\delta}{t},
\quad t\in(0,T],x\in D
$$
it follows that $ \Phi_\alpha (t_\alpha,x_\alpha,y_\alpha)\ge \theta_\delta\ge \bar\theta$, namely
\begin{equation}
    \label{basicthetabar}
u(t_\alpha,x_\alpha)-v(t_\alpha,y_\alpha)\ge
\bar\theta + \delta  (1 + |x_\alpha|^2)
+ \delta  (1 + |y_\alpha|^2)+\frac{\delta}{t_\alpha}
+\frac{\alpha}{2}|x_\alpha-y_\alpha|^2.
\end{equation}
Using 
\eqref{uvlineargrowth} once more, we deduce from 
\eqref{basicthetabar}
that 
$$ 
C_1(1+|x_\alpha|)
+C_1(1+|y_\alpha|)
\ge  \delta  (1 + |x_\alpha|^2)
+ \delta  (1 + |y_\alpha|^2)+\frac{\delta}{t_\alpha}
$$
which implies that there exists a constant $C$, independent from $\alpha$, such that
\begin{equation}
    \label{xyalphabdd}
 |x_\alpha|+ |y_\alpha|+\frac{1}{t_\alpha}\le C.
\end{equation}
Moreover, by standard arguments (see e.g. \cite{CrIsLi1992} Lemma 3.1 or Proposition 3.7) we have
\begin{equation}
    \label{alphaediffibis}
\alpha| x_\alpha-  y_\alpha|^2\to0,
\qquad | x_\alpha-  y_\alpha|\to0
\end{equation}
as $ \alpha\to\infty$. 
By \eqref{xyalphabdd} the  family $(x_\alpha,y_\alpha,t_\alpha)_\alpha$ is bounded, so it admits a limit point, necessarily  of the form $(\bar x,\bar x,\bar t)$ by
\eqref{alphaediffibis}. 
\eqref{xyalphabdd} also 
implies that  $\bar t>0$. Suppose that we had $\bar t=T$: then 
letting $\alpha\to\infty$ along a subsequence, by upper semicontinuity it follows from
\eqref{basicthetabar}
and \eqref{alphaediffibis} that
$u(T,\bar x)-v(T, \bar y)\ge
\bar\theta >0$, which contradicts the assumption 
\eqref{uvboundary}.
So we conclude that
$0<\bar t <T$ and it follows  that 
$(x_\alpha,y_\alpha,t_\alpha)\in (0,T)\times D\times D$
for infinitely many $\alpha\to\infty$. 

Next recall that
$(t_\alpha,x_\alpha,y_\alpha)$ was a maximum point of $\Phi_\alpha$, that we rewrite in the form
$$
\Phi_\alpha (t,x,y)=\tilde u(t,x)-\tilde v(t,y)- \varphi_\alpha(t,x,y),
$$
 where we define
$$ 
\tilde u(t,x)= u(t,x)-\delta (1+|x|^2),
\qquad
\tilde v(t, y)= v(y)+\delta (1+|y|^2),
\qquad
\varphi_\alpha(t,x,y)= \frac{\delta}{t}
+\frac{\alpha}{2}|x-y|^2.
$$
Since the quadratic terms are smooth, the parabolic sub/superjets are related as follows:
\begin{align}\label{gettiparab}&
\bar P ^{2,+} \tilde u(t, x)=
\bar P ^{2,+}   u(t, x)+(0, -2\delta x, -2\delta I),
\qquad
\bar P ^{2,-} \tilde v(t, y)=
\bar P ^{2,-}   v(t, y)+(0, 2\delta y, 2\delta I).
\end{align}
We wish to apply the  the Crandall-Ishii Lemma in the parabolic form: see \cite{CrIs1990}, or \cite{CrIsLi1992} Theorem 8.3.
For our equation, which is backward in time, it is convenient to check the required assumptions in the form stated in \cite{FlSobook} Theorem 6.1: we must show that for every $M>0$ there exists a constant $C(M)$ such that
$$(a,p,X)\in \bar P ^{2,+} \tilde u(t, x),\quad  |p|+|x|+|X|+ |\tilde u(t, x)| \le M\quad  \Longrightarrow\quad 
a\ge - C(M),
$$
$$(a,p,X)\in \bar P ^{2,-} \tilde v(t, y),\quad  |p|+|y|+|X|+ |\tilde v(t, y)| \le M\quad  \Longrightarrow\quad 
a\le  C(M).
$$
We check the first implication, the other one being similar. Assume $(a,p,X)\in \bar P ^{2,+} \tilde u(t, x)$. Then by 
\eqref{gettiparab} we have
$$
(a,p+2\delta x,X+2\delta I)\in \bar P ^{2,+} u(t, x)
$$
and since $u$ is a subsolution to 
\eqref{HJBparabmodif} we have
$$
-a + K\,u(t,x) + F(t,x, p+2\delta x,X+2\delta I) \le 0.
$$
Since $\tilde u \le u$, recalling 
\eqref{FLip}  we have
\begin{align*}
    a&\ge   K\,\tilde u(t,x) + F(t,x, p+2\delta x,X+2\delta I) 
    \\&\ge 
    K\,\tilde u(t,x) + 
    F(t,x, p,X) 
    - 2\delta C_0|x|^2
\end{align*}
and if 
$|p|+|x|+|X|+ |\tilde u(t, x)| \le M$
we obtain the required inequality $a\ge -C(M)$ setting 
$$
C(M)= KM +2\delta C_0M^2 + 
\sup\{|F(t,x,p,X)|\,:\, |p|+|x|+|X| \le M,\,t\in [0,T]
\}.
$$
After checking the required assumptions
we can now apply the Crandall-Ishii lemma and conclude that 
there exist $a,b\in\R$, $X,Y\in\cals(N)$ such that
$$
(a,\alpha( x_\alpha-  y_\alpha),X)\in \bar P ^{2,+}\tilde u(t_\alpha, x_\alpha),
\quad
(b,\alpha(  x_\alpha- y_\alpha),Y)\in \bar P ^{2,-}\tilde v(t_\alpha, y_\alpha),
$$
\begin{equation}
   \label{gettiter}
   a-b= (\varphi_\alpha)_t (t_\alpha, x_\alpha, y_\alpha)=
-\frac{\delta}{t_\alpha^2},
   \qquad
\left(
\begin{array}{cc}
    X & 0 \\
    0 & -Y
\end{array}\right)\le 
3\alpha\left(
\begin{array}{cc}
    I & -I \\
    -I & I
\end{array}\right).
\end{equation}
From \eqref{gettiparab} it follows that
$$
(a,\alpha( x_\alpha-  y_\alpha)+2\delta x_\alpha,X)\in \bar P ^{2,+} u(t_\alpha, x_\alpha),
\qquad
(b,\alpha(  x_\alpha- y_\alpha)-2\delta y_\alpha,Y)\in \bar P ^{2,-}v(t_\alpha, y_\alpha),
$$
and since $u,-v$ are subsolutions to 
\eqref{HJBparabmodif},
\begin{align*}&
-a+ K\, u(t_\alpha, x_\alpha)+F\Big(t_\alpha, x_\alpha ,\alpha(x_\alpha- y_\alpha)+2\delta x_\alpha ,X+2\delta I\Big)\le 0,
\\&
-b+K\, v(t_\alpha, y_\alpha)+F
\Big(t_\alpha, y_\alpha,\alpha(x_\alpha- y_\alpha)-2\delta y_\alpha,Y-2\delta I\Big)\ge 0.
\end{align*}
Subtracting these inequalities and recalling 
 the equality in  
   \eqref{gettiter}
 we obtain
 \begin{align*}
     \frac{\delta}{t_\alpha^2}+
  K\, (u(t_\alpha, x_\alpha) -v(t_\alpha, y_\alpha))
  \le
  &
F
\Big( t_\alpha, y_\alpha,\alpha(x_\alpha- y_\alpha)-2\delta y_\alpha,Y-2\delta I\Big)
\\&
-
F\Big(t_\alpha, x_\alpha ,\alpha(x_\alpha- y_\alpha)+2\delta x_\alpha ,X+2\delta I\Big).
 \end{align*}
 Using \eqref{FLip}, the right-hand side can be estimated from above by
 $$ 
F
\Big( t_\alpha, y_\alpha,\alpha(x_\alpha- y_\alpha),Y\Big)
\\
-
F\Big( t_\alpha,x_\alpha ,\alpha(x_\alpha- y_\alpha) ,X\Big)
+ 2\delta C_0|x_\alpha|^2 + 
2\delta C_0|y_\alpha|^2  .
$$
Inequality \eqref{basicthetabar} 
implies
$$
u(t_\alpha,x_\alpha)-v(t_\alpha,y_\alpha)\ge
\bar\theta + \delta  (1 + |x_\alpha|^2)
+ \delta  (1 + |y_\alpha|^2)
$$
so we can estimate the left-hand side from below and arrive at
\begin{align*}
 &   K\,
\bar\theta + K\,\delta  (1 + |x_\alpha|^2)
+K\, \delta  (1 + |y_\alpha|^2)
\\
&\qquad \le
F
\Big( t_\alpha, y_\alpha,\alpha(x_\alpha- y_\alpha),Y\Big)
-
F\Big(t_\alpha, x_\alpha ,\alpha(x_\alpha- y_\alpha) ,X\Big)
+ 2\delta C_0|x_\alpha|^2 + 
2\delta C_0|y_\alpha|^2  .
\end{align*}
Recall that
 $C_0$ was  a  constant    depending   only  on $K_0,N,d$. 
Choosing $K\ge 2C_0$ we obtain
$$ K\,
\bar\theta  \le
F
\Big( t_\alpha, y_\alpha,\alpha(x_\alpha- y_\alpha),Y\Big)
-
F\Big(t_\alpha,x_\alpha ,\alpha(x_\alpha- y_\alpha) ,X\Big).
$$
It is well known (see \cite{CrIsLi1992} Example 3.6) that the right-hand side can be estimated as follows:
$$
K\,
\bar\theta \le \omega(
 \alpha|x_\alpha- y_\alpha|^2 + |x_\alpha- y_\alpha|)
$$
for a modulus $\omega$ (i.e. a function $\omega:[0,\infty)\to [0,\infty)$ such that $\omega(0+)=0$) that only depends on the Lipschitz constants of  $\Sigma $, $b $, $f $, hence only  on $K_0,N,d$. Letting $\alpha\to \infty$ along a subsequence and recalling \eqref{alphaediffibis} we have 
$K\,
\bar\theta  \le 0 $ which is a contradiction.

\medskip
{\it Step II.} Now we consider the general case. For $K>0$ we define
$$
u_K(t,x)=e^{-K(T-t)}u(t,x),\qquad
v_K(t,x)=e^{-K(T-t)}v(t,x).
$$
It is easy to check that $u_K$ and $v_K$ are, respectively, sub- and super-solutions to the equation
$$-v_t+K\,v(t,x) +F_K(t,x,Dv,D^2v)=0,
$$ 
where $F_K$ is defined   by
$$
F_K(t,x,p,X)= -\sup_{a\in A}\bigg[
\frac{1}{2}\hbox{Trace}\Big(\Sigma(x,a,t )\Sigma(x,a,t)^TX\Big) +\langle p, b(x,a,t)\rangle +e^{-K(T-t)} f(x,a,t)
\bigg]  
$$
We note that this equation is of the form
 \eqref{HJBparabmodif}
and it safisfies inequality \eqref{FLip}
with the same constant $C_0$. Therefore the result of Step I applies and taking  $K\ge 2C_0$ we conclude that $u_K\le v_K$ and therefore $u\le v$. 
\qed

As in the infinite horizon case we arrive at the following characterization of the value function.

\begin{Theorem}
Suppose   that Assumptions \ref{settingprobabgen} and \ref{ipcoeffpb} hold  and that   $h,q,f$ are continuous in $t\in [0,T]$ uniformly in $a\in A$. Then the   value function $V$ for the problem \eqref{pfHf} is the unique  viscosity solution of the HJB equation \eqref{HJBparabolica} in the class of functions having with linear growth in $x$ uniformly in $t$.
\end{Theorem}

{\bf Proof.}  The linear growth condition is the second inequality in \eqref{Vtlinearlip}. Uniqueness follows from the previous result. The fact that $V$ is a viscosity solution is a standard result:  see e.g. 
\cite{Phbook09}
or 
\cite{FlSobook}.
\qed

\section{Dynamic programming equation: verification theorems}
\label{sec:verif}

In general, verification theorems state that if the HJB equation admits a classical solution, and some additional conditions are satisfied, then the solution coincides with the value function and an optimal control admits a feedback form. 
Looking at the HJB equations one may note that the second order part degenerates when $x$ approaches the boundary of $D$. Therefore we will present results where a classical solution is assumed to exist only in the the interior of $D$, denoted
$\overset{\circ}{D}
=\{x=(x_i)\in \R^N\,:\, x_i>0,\, i=1,\ldots,N\}$. The strict positivity result in Proposition \ref{investimeV}-{\it 1} will repeatedly play a role.

In this section we assume that Assumptions \ref{settingprobabgen} and \ref{ipcoeffpb}  are satisfied and we still denote by $V$ the value function of the separated problem. We present two results, for the parabolic and the elliptic case respectively.

\begin{Theorem}\label{verifparab}
Suppose that $v\in C^{1,2}([0,T]\times \overset{\circ}{D})$ satisfies equation 
\eqref{HJBparabolica}  on $[0,T]\times \overset{\circ}{D}$ and  the terminal condition 
\eqref{HJBparabolicaterminal} on $ \overset{\circ}{D}$, and has polynomial growth in $x$ uniformly in $t$. Then $v\ge V$.

Also assume that, for every $(t,x)\in [0,T]\times \overset{\circ}{D}$, the supremum in the equation is achieved  at a point $a=\widehat{a}(t,x)\in A$ for a measurable function $\widehat{a}:[0,T]\times \overset{\circ}{D}\to A $. Assume finally that, for every $t\in[0,T]$ and $x=(x_i)\in\overset{\circ}{D}$, the closed-loop equation
\begin{equation}
\label{closedloop}
\left\{
\begin{array}{rclr}
d\widehat\rho_s^i&=&\displaystyle \sum_j \widehat\rho_s^j \,q(\widehat{a}(s,\widehat{\rho}_s),s,j,i)\,ds
+ \widehat\rho_s^i\,h(i,\widehat{a}(s,\widehat{\rho}_s),s)\,dW_s,
& s\in[ t,T],\, i\in S,
\\
\widehat\rho_t^i&=&x_i,&
\end{array}\right.
\end{equation}
has an $\F^W$-adapted continuous solution  $\widehat\rho$.

Then the control process in feedback form
\[
\widehat{\alpha}_s=\widehat{a}(s,\widehat{\rho}_s), \qquad s\in [t,T],
\]
is optimal and $v$ coincides with the value function $V$. 

In particular, a solution to \eqref{closedloop} exists if, for every $i,j\in S$, the functions
\begin{align}\label{coeffloclip}
    x\mapsto 
    q(\widehat{a}(s,x),s,j,i),
    \qquad
x\mapsto  h(i,\widehat{a}(s,x),s),
\end{align}
are locally Lipschitz on $\overset{\circ}{D}$, uniformly in $s$.
\end{Theorem}

\noindent{\bf Proof.} The argument is classical, but  we sketch a proof in order to show that the behavior of the solution near the boundary of $D$ is irrelevant.
We introduce the controlled Kolmogorov operator
    \[
\call^a v(t,x)   =
\frac{1}{2}\sum_{ij}\partial_{ij}^2v(t,x )\,
x_ix_j\sum_{k=1}^dh_k(i,a,t)h_k(j,a,t)
+\sum_{ij} \partial_iv(t,x)x_jq(a,t,j,i) ,
\]
and we write the HJB equation \eqref{HJBparabolica} in the form
\[
 v_t(t,x)+\sup_{a\in A}\bigg[\call^a v(t,x)   
+ 
\sum_i
x_i\,f(i,a,t) 
\bigg]     =0,
 \qquad 
v(T,x)= \sum_i
x_i\,g(i). 
\]
Let us fix $t\in[0,T]$ and $x\in\overset{\circ}{D}$. Given an arbitrary control process $\alpha\in\cala$ let us   denote by $\rho$ the  corresponding solution to the equation  in  \eqref{pfHf}.  For every integer $k>0$ define the stopping times
\begin{align}
    \label{stopintD}
    T_k=\inf\{s\ge t\,:\,|\rho_s|>k\text{ \;or\; } dist(\rho_s,\partial D)|<1/k\},
\end{align}
where $dist(\cdot,\partial D)$ denotes the distance from the boundary $\partial D$ of $D$.  By the Ito formula we have
\[
v(T\wedge T_k,\rho_{T\wedge T_k})-v(t,x)=\int_t^{T\wedge T_k} [v_t(s,\rho_s)+\call^{\alpha_s}v(s,\rho_s)]\,ds +\int_t^{T\wedge T_k} \sum_{i}\partial_iv(s,\rho_s) \,
\rho_s^i\,h(i,\alpha_s,s)\,dW_s.
\]
By the choice of $T_k$, the processes $\{ \partial_iv(s,\rho_s) \ :\ s\in [t,T\wedge T_k]\}$ are bounded, and the function $h$ is also assumed to be bounded. Therefore, upon taking expectation, the stochastic integral disappears. Summing and substracting terms, after rearrangement we obtain
\begin{align*}
v(t,x)&= \E\int_t^{T\wedge T_k} \Big\{-v_t(s,\rho_s)-\call^{\alpha_s}v(s,\rho_s)-\sum_{i}
\rho_s^i\,f(i,\alpha_s,s)\Big\}\,ds 
\\&\quad +\E\, 
[v(T\wedge T_k,\rho_{T\wedge T_k})]+\E\, \int_t^{T\wedge T_k} \sum_{i}
\rho_s^i\,f(i,\alpha_s,s)\,ds.
\end{align*}
By the strict positivity result in Proposition \ref{investimeV}-{\it 1}, the trajectories of $\rho$ never leave $\overset{\circ}{D}$, a.s. It follows that a.s.
we have $T\wedge T_k= T$ for large $k$. Letting $k\to\infty$ in the last displayed formula, by dominated convergence, the last two terms tend to
\begin{align*}
\E\, 
[v(T,\rho_{T})]+\E\, \int_t^{T} \sum_{i}
\rho_s^i\,f(i,\alpha_s,s)\,ds=J_T(t,x,\alpha).
\end{align*}
By the HJB equation the term $\{\ldots\}$ is nonnegative,  and by  monotone convergence we obtain
\begin{align*}
v(t,x)&= \E\int_t^{T} \Big\{-v_t(s,\rho_s)-\call^{\alpha_s}v(s,\rho_s)-\sum_{i}
\rho_s^i\,f(i,\alpha_s,s)\Big\}\,ds +J_T(t,x,\alpha).
\end{align*}
Since $\{\ldots\}\ge0$ it follows that $v(t,x)\ge J_T(t,x,\alpha)$ for every $\alpha\in\cala$ and therefore $v(t,x)\ge V(t,x)$. When the control $\widehat{\alpha}$ is chosen we have $\{\ldots\}=0$ and it follows that $v(t,x)=J_T(t,x,\widehat{\alpha})$, which shows the optimality of  $\widehat{\alpha}$ and the equality $v(t,x)=V(t,x)$.

To prove the final statement of the Theorem we note that   the closed-loop equation \eqref{closedloop} has coefficients with linear growth in $x$ and, when \eqref{coeffloclip} holds, also locally Lipschitz, and therefore a unique solution exists up to the stopping times $T_k\wedge T$. As noted above, by Proposition \ref{investimeV}-{\it 1}, a.s. we have $T\wedge T_k= T$ for large $k$,  so that the solution exists on the whole interval $[0,T]$. 
\qed

\begin{Theorem}
Suppose that the coefficients $q$, $h_k$, $f$ do not depend on time and that $v\in C^{2}( \overset{\circ}{D})$ satisfies equation 
\eqref{HJBellittca} on $\overset{\circ}{D}$. Suppose that for every controlled trajectory $\rho$ starting at $x\in \overset{\circ}{D}$ we have 
\[
\lim_{T\to\infty} e^{-\beta T}\E\,[v(\rho_T)]=0.
\]
 Then $v\ge V$.

Also assume that, for every $x\in  \overset{\circ}{D}$, the supremum in the equation is achieved  at a point $a=\widehat{a}(x)\in A$ for a measurable function $\widehat{a}: \overset{\circ}{D}\to A $. Assume finally that for every $x=(x_i)\in\overset{\circ}{D}$, the closed-loop equation
\begin{equation}
\label{closedloopellittica}
\left\{
\begin{array}{rclr}
d\widehat\rho_s^i&=&\displaystyle \sum_j \widehat\rho_s^j \,q(\widehat{a}(\widehat{\rho}_s),j,i)\,ds
+ \widehat\rho_s^i\,h(i,\widehat{a}(\widehat{\rho}_s))\,dW_s,
& s\ge 0,\, i\in S,
\\
\widehat\rho_0^i&=&x_i,&
\end{array}\right.
\end{equation}
has an $\F^W$-adapted continuous solution  $\widehat\rho$.

Then the control process
\[
\widehat{\alpha}_s=\widehat{a}(\widehat{\rho}_s), \qquad s\ge 0,
\]
is optimal and $v$ coincides with the value function $V$. 

In particular, a solution to \eqref{closedloopellittica} exists if, for every $i,j\in S$, the functions
\begin{align}
    x\mapsto 
    q(\widehat{a}(x),j,i),
    \qquad
x\mapsto  h(i,\widehat{a}(x)),
\end{align}
are locally Lipschitz on $\overset{\circ}{D}$.
\end{Theorem}

\noindent{\bf Proof.} We only sketch the arguments, which are similar to the previous ones. Let $\rho$ denote the trajectory corresponding to an arbitrary control $\alpha$ and starting point $x\in \overset{\circ}{D}$. By Proposition \ref{investimeV}-{\it 1}, $\rho$  never hits the boundary of $D$, a.s. Applying the Ito formula to $e^{-\beta s}v(\rho_s)$ on $[0,T\wedge T_k]$, taking expectation and letting $k\to\infty$ and $T\to\infty$  we obtain
\begin{align*}
v(x)&= \E\int_0^{\infty}e^{-\beta s} \Big\{\beta v(\rho_s)-\call^{\alpha_s}v(\rho_s)-\sum_{i}
\rho_s^i\,f(i,\alpha_s)\Big\}\,ds +J_\infty(x,\alpha).
\end{align*}
As before, the term in curly brackets  is nonnegative and it is zero when $\alpha=\hat\alpha$. The conclusion follows.
\qed

\begin{Example} \emph{
Consider the case when $A\subset\R$ is an interval $[0,R]$ for some $R>0$. Take
\[
q(a,t,i,j)=a, \quad h(i,a,t)=h(i), \quad
f(i,a,t)=-\frac{a^2}{2}
\]
for $a\in[0,R]$, $t\in [0,T]$, $i,j\in S$. Thus, we are considering a control problem for a Markov chain $X$ with controlled transition rates that can take any value in $[0,R]$ and reward functional and observation process given by
$$
J(\alpha)=\bar\E\left[ -\frac{1}{2}\int_0^T 
\alpha_t^2\,dt   + 
g(X_T)\right],
\qquad
\int_0^th(X_s)\,ds+B_t,
$$
for arbitrary $g:S\to\R$.
Setting
$\gamma_{ij}= \sum_{k=1}^dh_k(i)h_k(j)$, the HJB equation \eqref{HJBparabolica} becomes 
\[
 v_t(t,x) +\frac{1}{2}\sum_{ij}\partial_{ij}^2v(t,x )\,
x_ix_j \gamma_{ij}
  +\sup_{a\in [0,R]}\bigg[a\,
\sum_{ij} \partial_iv(t,x)x_j
-\frac{a^2}{2}
\,\sum_i
x_i 
\bigg]     =0,
\]
with the boundary condition $
 v(T,x)= \sum_i
x_i\,g(i)$. Setting
\[\widehat a(p):=
\underset{a\in [0,R]}{\mathrm{arg\,max}}
\left[p\, a-\frac{a^2}{2}\right]=p^+\wedge R, \qquad p\in\R,
\]
the equation becomes
\begin{equation*}
 v_t(t,x) +\frac{1}{2}\sum_{ij}\partial_{ij}^2v(t,x )\,
x_ix_j \gamma_{ij}
  + \widehat a\left( 
\frac{\sum_{ij} \partial_iv(t,x)x_j}
{ \sum_i
x_i} 
\right)\sum_i
x_i     =0.
\end{equation*}
Assume that a solution $v\in C^{1,2}([0,T]\times \overset{\circ}{D})$ exists, with polynomial growth in $x$ uniformly in $t$. Then Theorem \ref{verifparab} applies and we conclude that the closed-loop equation
\[
\left\{
\begin{array}{rclr}
d\widehat\rho_s^i&=&\displaystyle \widehat a\left( 
\frac{\sum_{\ell j} \partial_\ell v(s,\widehat \rho_s)\widehat \rho_s^j}
{ \sum_\ell
\widehat\rho_s^\ell} 
\right) \sum_j \widehat\rho_s^j \,ds
+ \widehat\rho_s^i\,h(i)\,dW_s,
& s\in[ t,T],\, i\in S,
\\
\widehat\rho_t^i&=&x_i,&
\end{array}\right.
\]
has an $\F^W$-adapted continuous solution  $\widehat\rho$, the control process in feedback form
\[
\widehat{\alpha}_s=\widehat a\left( 
\frac{\sum_{\ell j} \partial_\ell v(s,\widehat \rho_s)\widehat \rho_s^j}
{ \sum_\ell
\widehat\rho_s^\ell} 
\right), \qquad s\in [t,T],
\]
is optimal and $v$ coincides with the value function.  \qed
}
\end{Example}

\begin{Remark}\label{equnifellitt} 
\emph{ 
Define  $\gamma_{ij}(a)= \sum_{k=1}^dh_k(i,a)h_k(j,a)$ and write the equation  \eqref{HJBellittca} in the form
\begin{equation}
\label{HJBellittcaprovv}
\beta v(x)   -\sup_{a\in A}\bigg[
\frac{1}{2}\sum_{ij}\partial_{ij}^2v(x )\,x_ix_j\,\gamma_{ij}(a)
+\sum_{ij} \partial_iv(x)x_jq(a,j,i)
+ 
\sum_i
x_i\,f(i,a)  
\bigg]     =0.
\end{equation}
Assume in addition that the functions $h(i, \cdot):A \rightarrow \mathbb{R}^d$, $q(\cdot,i,j):A \rightarrow [0,\infty)$ and $f(i, \cdot) : A \rightarrow \mathbb{R}$ are continuous for every $i,j \in S$ and the following ellipticity condition holds: there exists $\kappa>0$ such that
\[
\sum_{i,j}\gamma_{ij}(a)\xi_i\xi_j\ge \kappa\,|\xi|^2, \qquad \xi\in\R^N,\, a\in A.
\]
Note that this may happen only provided $N\le d$.
Then one may prove that the solution $v$ is in fact of class $C^2(\overset{\circ}{D})$ with H\"older continuous second derivatives and it satisfies the equation in the classical sense. 
This follows from a result in \cite{Saf88}, established for bounded smooth domains in
$\mathbb{R}^n$ and thus applies to any smooth domain compactly contained in $\overset{\circ}{D}$. (In this reference the supremum is taken over a countable family; the extension to our setting is a direct consequence of the continuity of the coefficients with respect to 
$a$ and the fact that the control action space $A$ is Polish): The same result can be achieved as in \cite{GoVa02} by a logarithmic change of variables $y_i=\log x_i$
introducing the 
auxiliary unknown function
$$
w(y_1,\ldots,y_N)=v(e^{y_1} ,\ldots,e^{y_N}),
\qquad y=(y_1,\ldots,y_N)\in\R^N,
$$
which is defined on $\R^N$. 
Similarly, in the parabolic case, the equation \eqref{HJBparabolica} can be written as
\begin{equation}
\label{HJBparabolica_2}
    - v_t(t,x)   -\sup_{a\in A}\bigg[
\frac{1}{2}\sum_{ij}\partial_{ij}^2v(t,x )\,
x_ix_j\gamma_{ij}(a,t)
+\sum_{ij} \partial_iv(t,x)x_jq(a,t,j,i)
+ 
\sum_i
x_i\,f(i,a,t) 
\bigg]     =0,
\end{equation}
where $\gamma_{ij}(a,t)= \sum_{k=1}^dh_k(i,a,t)h_k(j,a,t)$.
Assume that the functions $h(i, \cdot,t):A \rightarrow \mathbb{R}^d$, $q(\cdot,t,i,j):A \rightarrow [0,\infty)$ and $f(i, \cdot) : A \rightarrow \mathbb{R}$ are continuous for every $i,j \in S$, $ t\in [0,T]$ and the following   ellipticity condition holds: there exists $\kappa>0$ such that
\[
\sum_{i,j}\gamma_{ij}(a,t)\xi_i\xi_j\ge \kappa\,|\xi|^2, \qquad \xi\in\R^N,\, a\in A, \, t \in [0,T].
\]
One may prove again that the solution $v$ is of class $C^{1,2}([0,T]\times \overset{\circ}{D})$ with  H\"older continuous derivatives. This follows from \cite{Wang90}, Theorem 1.1, see also \cite{CrKoSw00} (comments before Theorem 9.1).   \qed
}
\end{Remark}

\section{Stochastic maximum principle}
\label{secSMP}

We devote this  final section to formulate a stochastic maximum principle for the separated problem as a necessary condition for optimality. Although the proof partially relies on known results, the formulation improves existing results in the literature, especially because we remove the restriction that the set of control  actions $A$ should be convex and the coefficients be differentiable with respect to $a\in A$. 

In this section we assume that Assumptions \ref{settingprobabgen} and \ref{ipcoeffpb} hold, and for simplicity we only treat the finite horizon case starting at time $0$, namely:
$$
d\rho_t^i= \sum_j \rho_t^j q(\alpha_t,t,j,i)\,dt
+ \rho_t^i\,h(i,\alpha_t,t)\,dW_t,\qquad \rho_0^i=x_i,
\qquad t\in[0,T],\, i\in S,
$$
$$
J(\alpha)=\E\left[ \int_0^T  \sum_i
\rho_t^i\,f(i,\alpha_t,t)\,dt   + \sum_i\rho_T^i\,g(i)
\right].
$$

It is convenient to write it in vector form. Let us recall the definition of the matrix $Q^a$ and let us define the $N$-dimensional vectors
$$
\rho=(\rho^i)_i,
\quad
h_k(a,t)=(h_k(i,a,t))_i,
\quad
f(a,t)= (f(i,a,t) )_i,
\quad
g=(g(i))_i
$$
for $a\in A,$ $t\ge0$, $k=1,\ldots,d$.
We also denote the componentwise multiplication of vectors as follows: $$x*y= (x(i)y(i))_i, \qquad \hbox{for}\quad x=(x(i))_i,\, y=(y(i))_i\in\R^N.
$$
With this notation we write
$$
d\rho_t= (Q^{\alpha_t}_t)^{\rm T} \rho_t \,dt
+ \sum_{k=1}^d\rho_t *  h_k(\alpha_t,t)\,dW^k_t,
\qquad
J(\alpha)=\E\left[ \int_0^T  \langle
\rho_t,f(\alpha_t,t)\rangle\,dt + \langle\rho_T,g\rangle
\right],
$$
where $\langle\cdot,\cdot\rangle$ stands for the scalar product in $\R^N$.

For an arbitrary admissible control $(\alpha_t)$   we consider the adjoint BSDE
\begin{equation}
    \label{adjBSDE}
    \left\{
    \begin{array}{rcl}  
       -dp_t&=  & \displaystyle
       -  \sum_{k=1}^d   q^k _t\,dW^k_t+ \Big(
Q_t^{\alpha_t}p_t+ \sum_{k=1}^d h_k(\alpha_t,t)*  q^k _t + f (\alpha_t,t)
\Big)\,dt,
\\
       p_T&=  & g .
    \end{array}
    \right.
\end{equation}
In our specific case the controlled trajectory $(\rho_t)$ does not occur in the BSDE. 
The solution is understood in the usual way: the process $p$ is continuous adapted, the processes $q^1,\ldots,q^d$ are progressive,  and
$$\E\left[\sup_{t\in [0,T]}|p_t|^2+\sum_{k=1}^d\int_0^T|q^k_t|^2\right]<\infty.
$$
Within this class there exists a solution $(p,q^k)$, the process $p$ is unique up to indistinguishability and the processes $q^k$ up to equality $d\P\otimes dt$-a.s. This follows from standard results on BSDEs and our boundedness assumptions on $q(a,t,i,j)$, $h_k(i,a,t)$, $f(i,a,t)$. 

\begin{Theorem}\label{thSMP} Suppose that the coefficients $q$, $h_k$, $f$ are continuous functions of $a\in A$, for fixed $t,i,j$.  
Assume that $(\alpha_t)$ is an optimal control. 
    Let $(\rho_t)$ be the corresponding trajectory  and $(p_t,q^k_t)$ the solution to the adjoint BSDE. Defining the Hamiltonian
    $$
    H(t,\rho,a,p,q^1,\ldots, q^d)=
     f (a,t) + 
    \langle Q_t^a p,\rho\rangle +
\sum_{k=1}^d
    \langle q^k,h_k(a,t)*\rho\rangle,
    \qquad a\in A;\;  \;\rho,p,q^k\in\R^N
    $$
    we have, $d\P\otimes dt$-a.s.,
$$H(t,\rho_t,\alpha_t,p_t,q^1_t,\ldots, q^d_t)=
    \max_{a\in A} H(t,\rho_t,a,p_t,q^1_t,\ldots, q^d_t).
$$
\end{Theorem}

{\bf Proof.} As explained before, this result is essentially an application of the general stochastic maximum principle in \cite{Pe1990} (see also \cite{YoZhbook} for a careful exposition). Our sketch of proof is simply intented to give the reader exact indications for all details and  warn about the minor changes required in our case.

Take an arbitrary admissible control $(\bar \alpha_t)$. For any $\epsilon \in(0,T]$ and any Borel set $I_\epsilon\subset [0,T]$ with Lebesgue measure   $|I_\epsilon|=\epsilon$,  define the spike variation control setting
$$
\alpha^\epsilon_t=\left\{ 
\begin{array}{ll}
\alpha_t,& t\in {I_\epsilon },
\\
\bar\alpha_t,& t\in [0,T]\backslash {I_\epsilon }.
\end{array}
\right.
$$
Since $\alpha$ is optimal we have $J(\alpha^\epsilon)\le J(\alpha)$. Proceeding   as in \cite{Pe1990} one arrives at
\begin{equation}
    \label{459}
0\ge J(\alpha^\epsilon)-J(\alpha)
= \E\left[
\int_0^T
\Big(H(t,\rho_t,\alpha_t^\epsilon,p_t,q^1_t,\ldots, q^d_t) -
H(t,\rho_t,\alpha_t,p_t,q^1_t,\ldots, q^d_t)
\Big)\,dt 
\right] +o(\epsilon).
\end{equation}
This follows immediately from
formula (4.59) in Section 5.4 of \cite{YoZhbook}, and
the reader may find a detailed proof there. In fact, since $H$ and the terminal reward $\langle g,\rho\rangle$ are linear functions of $\rho$,  their second derivatives with respect to $\rho$ vanish and the 
formula (4.59) in \cite{YoZhbook} reduces to \eqref{459}.

We describe in some detail how the conclusion follows from \eqref{459}, a point which is often neglected in several papers. We follow the elegant approach of \cite{LiTang94}, which    is based on the following result (compare Lemma 2.2 of \cite{LiTang94}) and avoids using the Lebesgue differentiation theorem.

\begin{Lemma}
Let $\ell:[0,T]\to\R$ be Borel measurable and satifying $\|\ell\|_{L^1}:=\int_0^T|\ell(t)|\,dt<\infty$. Then for any $\epsilon \in(0,T]$ there exists a Borel set $I_\epsilon\subset [0,T]$ with   $|I_\epsilon|=\epsilon$ and such that 
\[
\left|\frac{\epsilon}{T} \int_0^T\ell(t)\,dt-\int_{I_\epsilon}\ell(t)\,dt
\right|\le \epsilon^2.
\]
\end{Lemma}

\noindent {\bf Proof.}
We present a self-contained and simplified proof of a more general result that can be found in Theorem 2 in \cite{LiYao85}. 
Take a finite-valued function $\bar\ell$ such that $\|\ell-\bar\ell\|_{L^1}\le \epsilon^2/2$. Write $\bar\ell$ in the form $\sum_{i=1}^n\ell^i\,1_{E^i}$ for $\ell^i\in\R$ and a finite partition $\{E^i\}$ of   $[0,T]$ consisting of Borel sets. Since the Lebesgue measure is non-atomic, there exist Borel sets $E^i_\epsilon\subset E^i$ such that $|E^i_\epsilon|=\epsilon |E^i|/T$. Then we have
\[
\frac{\epsilon}{T} \int_0^T\bar\ell(t)\,dt=
\sum_{i=1}^n\ell^i\,\epsilon\,|E^i|/T= \sum_{i=1}^n\ell^i\,\,|E^i_\epsilon|=\int_{I_\epsilon}\bar\ell(t)\,dt
,
\]
provided we set $I_\epsilon=\cup_{i=1}^nE^i_\epsilon$. We have $|I_\epsilon|=\sum_{i=1}^n|E^i_\epsilon|=\sum_{i=1}^n\epsilon|E^i|/T=\epsilon$ and
\begin{align*}
\left|\frac{\epsilon}{T} \int_0^T\ell(t)\,dt-\int_{I_\epsilon}\ell(t)\,dt
\right|
&=
\left|\frac{\epsilon}{T} \int_0^T\ell(t)\,dt-\frac{\epsilon}{T} \int_0^T\bar \ell(t)\,dt+\int_{I_\epsilon}\bar\ell(t)\,dt-\int_{I_\epsilon}\ell(t)\,dt
\right| 
\\&\quad 
\le \frac{\epsilon}{T}\|\ell-\bar\ell\|_{L^1} +\|\ell-\bar\ell\|_{L^1}\le \epsilon^2 .
\qed
\end{align*}

We conclude the proof of Theorem \ref{thSMP}. Apply the previous lemma to the function 
\[\ell(t)=\E\left[
H(t,\rho_t,\bar\alpha_t,p_t,q^1_t,\ldots, q^d_t) -
H(t,\rho_t,\alpha_t,p_t,q^1_t,\ldots, q^d_t) 
\right]
\]
and choose the set $I_\epsilon$ accordingly. Then
\eqref{459} yields
$\int_{I_\epsilon}\ell(t)\,dt\le o(\epsilon)$. By the lemma we also have $\frac{\epsilon}{T}\int_0^T\ell(t)\,dt\le o(\epsilon)$ and we conclude that
\begin{equation}
\label{Banegl}
\int_0^T\ell(t)\,dt
= \E\left[
\int_0^T
\Big(H(t,\rho_t,\bar\alpha_t,p_t,q^1_t,\ldots, q^d_t) -
H(t,\rho_t,\alpha_t,p_t,q^1_t,\ldots, q^d_t)
\Big)\,dt 
\right] \le 0
\end{equation}
for an arbitrary admissible control $\bar\alpha$.

Given any $a\in A$, let
\[
B_a=\{(\omega,t)\in\Omega\times [0,T]: 
H(t,\rho_t(\omega),a,p_t(\omega),q^1_t(\omega),\ldots, q^d_t(\omega)) >
H(t,\rho_t(\omega),\alpha_t(\omega),p_t(\omega),q^1_t(\omega),\ldots, q^d_t(\omega))\}.
\]
Choosing
$\bar\alpha_t(\omega)=a\,1_{B_a}(\omega,t)+\alpha_t(\omega)\,1_{(\Omega\times [0,T])\backslash B_a}(\omega,t)$ in  \eqref{Banegl} it follows that $B_a$ is $d\P\otimes dt$-negligible. 
In other words, for every $a\in A$,
$$H(t,\rho_t,\alpha_t,p_t,q^1_t,\ldots, q^d_t)\ge H(t,\rho_t,a,p_t,q^1_t,\ldots, q^d_t), \qquad d\P\otimes dt-a.s.
$$
By choosing a countable dense set of $a$'s in $A$, and using the continuity of the coefficients with respect to $a$, we obtain the required conclusion.
\qed

\bigskip

{\bf Acknowledgements.}  The authors wish to thank Prof. Andrzej \'Swi\c{e}ch for his help and suggestions on viscosity solutions to the dynamic programming equations considered in this paper.

\bibliographystyle{plain}
\bibliography{biblio}

\end{document}